\documentclass{amsart}
\usepackage{amssymb, amsmath, graphicx, amsthm,amsfonts}
\usepackage[utf8]{inputenc}
\usepackage{tikz}
\usepackage{tikz-cd}
\usetikzlibrary{shapes.geometric}
\usetikzlibrary{cd}
\usetikzlibrary{graphs,decorations.pathmorphing,decorations.markings}

\usepackage{enumerate} 

\usepackage[colorlinks = true,
            linkcolor = blue,
            urlcolor  = red,
            citecolor = red,
            anchorcolor = blue,
            urlbordercolor= red ]{hyperref}
            
\newtheorem{theorem}{Theorem}[section]
\newtheorem{lemma}[theorem]{Lemma}
\newtheorem{corollary}[theorem]{Corollary}

\newtheorem{prop}[theorem]{Proposition}
\newtheorem{definition}[theorem]{Definition}
\newtheorem{example}[theorem]{Example}
\newtheorem{remark}[theorem]{Remark}

\let\div\relax
\DeclareMathOperator{\div}{div}
\newcommand{\bam}[1]{\textbf{#1}}
\DeclareMathOperator{\Aut}{Aut}
\DeclareMathOperator{\Res}{Res}
\newcommand{\dif}{\mathrm{d}}
\DeclareMathOperator{\id}{Id}
\DeclareMathOperator{\Jac}{Jac}

\newcommand{\pround}[1]{\left( #1 \right)}
\newcommand{\psquare}[1]{\left[ #1 \right]}
\newcommand{\pbrace}[1]{\left\{ #1 \right\} }
\newcommand{\pangle}[1]{\left\langle #1 \right\rangle}

\usepackage{caption}
\usepackage{subcaption}
\usepackage{array}
\newcolumntype{P}[1]{>{\centering\arraybackslash}p{#1}}

\usepackage{wasysym}

\usepackage{float}
\title{Bring's curve: Old and New}

\author{H. W. Braden and Linden Disney-Hogg}

\address{
School of Mathematics and Maxwell Institute for Mathematical Sciences\\ The University of Edinburgh\\ 
Edinburgh EH9 3FD, Scotland, U.K.
}

\email{hwb@ed.ac.uk, A.L.Disney-hogg@sms.ed.ac.uk}

\begin{document}
\maketitle

\begin{abstract}
    Bring's curve, the unique Riemann surface of genus-4 with automorphism group $S_5$, has many exceptional properties.
    We review, give new proofs of, and extend a number of these including giving the complete realisation of the automorphism group for a plane curve model, identifying a new elliptic quotient of the curve and the modular curve $X_0(50)$, providing a complete description of the orbit decomposition of the theta characteristics, and identifying the unique invariant characteristic with the divisor of the Sz\"ego kernel. In achieving this we have used modern computational tools in Sagemath, Macaulay2, and Maple, for which notebooks demonstrating calculations are provided.

\end{abstract}
%%%%%%%%%%%%%%%%%%%%%%%%%%%%%%%%%%%%%%%%%%%%%%%%%%%%%%%%
%%%%%%%%%%%%%%%%%%%%%%%%%%%%%%%%%%%%%%%%%%%%%%%%%%%%%%%%
%%%%%%%%%%%%%%%%%%%%%%%%%%%%%%%%%%%%%%%%%%%%%%%%%%%%%%%%
%%%%%%%%%%%%%%%%%%%%%%%%%%%%%%%%%%%%%%%%%%%%%%%%%%%%%%%%
\section{Introduction}
Bring's curve, a genus-4 Riemann surface first introduced by Erland Bring in 1786 in relation to solutions of the quintic equation, is the unique genus-4 curve with automorphism group $S_5$, the largest possible group for curves of this genus. This automorphism group acts transitively on the Weierstrass points of the curve, which are all weight one, a very rare property of curves. The curve can moreover be seen to be an example of an elliptic modular surface, and has a unique (even) theta characteristic invariant under the action of the whole automorphism group. It shares all these properties with Klein's curve, a genus-3 curve which has been studied extensively \cite{Levy2001}. 

We aim to give a unified review of many of these aspects of Bring's curve, extend these results and provide connections between them.
This will include providing an explicit presentation of the $S_5$ group in terms of birational automorphisms of a plane-curve model, using this to completely describe the quotient structure of the curve, identifying fixed points of these quotients with geometrically significant points on the curve including the Weierstrass points, and utilising these geometric points to describe the orbit decomposition of odd theta characteristics on the curve. 

Moreover, we will construct an explicit birational map between the canonical embedding of the curve and a plane-curve model. This will allow us to use modern tools in SageMath, Maple, and Macaulay2 (all interfaced through Sage) to clarify and fix errors in the existing literature as well as finding the orbit decomposition of even theta characteristics of the curve. We shall also calculate the unique invariant theta characteristic on the curve as a vector in the Jacobian, making explicit previous work of others. These calculations are presented in Jupyter notebooks, available at \href{https://github.com/DisneyHogg/Brings_Curve}{\tt{https://github.com/DisneyHogg/Brings\_Curve}}. 

The structure of the papers is as follows, in \S\ref{sec: defining equations} we recall basic properties from the literature, provide the birational map between models of the curve and describe the period matrix and automorphism group of the curve. In \S\ref{sec: geometric points} we identify geometrically relevant orbits of points under the action of the automorphism group, notably the Weierstrass points $W$, for which we construct the associated holomorphic differentials with vanishing of order $4$ at $W$ and the meromorphic functions with poles only at $W$; indeed we shall describe the divisors of these completely. (Recall, for comparison, that while the pole of the Weierstrass $\wp$-function
is $z=0$, describing its zeros is nontrivial.)
In \S\ref{sec: quotients by subgroups} we use the action of the automorphism group to explicitly describe the quotient structure of the curve.
Using representation theory we elucidate much of the structure of the Jacobian of Bring's curve, giving new derivations of previously known results, and identifying new quotients of the curve. In doing so, we also discover isomorphisms (as opposed to isogenies) of quotients not explained by group theory alone, and these deserve further study. In \S\ref{sec: theta characteristics} we will develop the understanding of the theta characteristics on the curve, including explicitly the unique invariant one. 
Bring's curve has arisen in many disparate mathematical areas and we will use this paper to also make some of these results and interconnections more widely known; we will however not focus on either the modular aspects of Bring's curve or extension to positive characteristic here. 
Because of the wide-ranging nature of this paper we will conclude in \S\ref{sec: conclusion} by summarising our new results having by then placed these in context. We draw attention here to two areas we feel merit further explanation.

This work answers questions first raised in \cite{Braden2012}, and we will leave many expository results to that paper. For a source on background classical material, see \cite{Farkas1992}. Unless otherwise said explicitly, we shall assume the Riemann surfaces we discuss are all smooth. 

\vspace{0.1in}
\noindent\textbf{Acknowledgements.} We are grateful to Vanya Cheltsov for discussions. The research of LDH is supported by a UK Engineering and Physical Sciences Research Council (EPSRC) studentship. 

%%%%%%%%%%%%%%%%%%%%%%%%%%%%%%%%%%%%%%%%%%%%%%%%%%%%%%%%
%%%%%%%%%%%%%%%%%%%%%%%%%%%%%%%%%%%%%%%%%%%%%%%%%%%%%%%%
%%%%%%%%%%%%%%%%%%%%%%%%%%%%%%%%%%%%%%%%%%%%%%%%%%%%%%%%
%%%%%%%%%%%%%%%%%%%%%%%%%%%%%%%%%%%%%%%%%%%%%%%%%%%%%%%%
\section{Defining Equations and Elementary Properties}\label{sec: defining equations}
%%%%%%%%%%%%%%%%%%%%%%%%%%%%%%%%%%%%%%%%%%%%%%%%%%%%%%%%
%%%%%%%%%%%%%%%%%%%%%%%%%%%%%%%%%%%%%%%%%%%%%%%%%%%%%%%%

In this section we shall introduce Bring's curve describing its basic properties, its period matrix and its automorphisms. We begin with two descriptions: one as a model in $\mathbb{P}^4$ and the second as a plane curve. In particular we shall construct the birational map between these models, which we believe new. Such a map enables us to concretely translate calculations from one model to another and will be helpful later in \S\ref{sec: quotients by subgroups} when we look at quotients of the curve. 

\begin{definition}
\bam{Bring's Curve $\mathcal{B}$} is defined in $\mathbb{P}^4$ by the (homogeneous) equations\footnote{Note here we use the notation $H_k$ of \cite{Burns1983} for the symmetric sum, whereas \cite{Edge1978} uses $S_k$.}
\begin{equation}\label{eq: P4 model of Bring}
H_k := \sum_{i=1}^5 x_i^k = 0  , \quad k=1, 2, 3, 
\end{equation}
where we have taken the coordinates $[x_1 : x_2 : x_3 : x_4 : x_5] \in \mathbb{P}^4$. 
\end{definition}
We will call (\ref{eq: P4 model of Bring}) the
$\mathbb{P}^4$-model of Bring's curve.
Historically, this curve was attributed to Bring by Klein in his lectures on the icosahedron \cite{Klein1888}. Bring's improvement on the Tschirnhaus transformation enabled the general quintic to be brought to the Bring-Jerrard form $x^5 + ex+f=0$ \cite{Adamchik2003}: if $x_i$, $i=1,\dots, 5,$ are the roots of this equation then the symmetric sums $\sum_i x_i^k=0$ for $k=1, 2, 3$, and so the above curve emerges. Green \cite{Green1978} exploits this relation to Bring's curve to (implicitly) solve the quintic using automorphic forms. At this stage we may observe that $S_5\leq \Aut(\mathcal{B})$.

\begin{remark}
In the above definition we have implicitly taken $\mathbb{P}^4 = \mathbb{CP}^4$, and indeed this will be our definition throughout. At various junctures we will however highlight the number fields within $\mathbb{C}$ that are relevant fields of definition for morphisms of the curve, or geometrically important points on the curve.
\end{remark}

For many calculations a plane model of a curve in $\mathbb{P}^2$ is useful.
One of the main tools we will utilise in this paper is the Riemann surfaces module of SageMath \cite{sagemath2021} which requires such a model. As such we introduce the following.
\begin{definition}
The \bam{Hulek-Craig (HC) model} of Bring's curve is the (singular) plane curve in $\mathbb{P}^2$ given by 
\begin{equation}\label{eq: P2 model of Bring}
  F(X,Y,Z) :=
    X(Y^5+Z^5) + (XYZ)^2 - X^4YZ -2(YZ)^3  = 0  ,
\end{equation}
taking homogeneous coordinates $[X:Y:Z] \in \mathbb{P}^2$.
\end{definition}
This model was used in \cite{Hulek1985, Craig2002} where they studied the curves modular properties. Indeed Craig gives a parameterization of the curve in terms of modular functions following work of Ramanujan. Craig does not identify the curve as Bring's, whereas Hulek finds the equation as a hyperplane section of a surface that parameterises elliptic normal curves of degree 5, and then refers to previous results to connect it with Bring. This connection was first made in \cite{Naruki1978}, where the curve is attributed to Klein. Indeed, because of this interpretation, one can write down Bring's curve as 
\[
\overline{\mathcal{H}/\Gamma_0(2, 5)} \quad \text{where} \quad \Gamma_0(2,5) = \Gamma_0(2) \cap \Gamma(5) =  \left\lbrace\begin{pmatrix} a & b \\ c & d \end{pmatrix} \in SL_2(\mathbb{Z}) \, \big\vert \, \begin{array}{l} a \equiv d \equiv 1\pmod 5 \\ \phantom{d \equiv } \;  b \equiv 0  \pmod 5 \\ \phantom{d \equiv } \; c \equiv 0  \pmod{10} \end{array} \right \rbrace . 
\]
To our knowledge, \cite{Hulek1985} is the first occurrence of this plane model of the curve; Klein will write it down in \cite[p. 165, p. 242]{Klein1888} but the relation to a plane model of Bring's curve is not made explicit. We will call (\ref{eq: P2 model of Bring}) the HC-model of Bring's curve. In the following subsection we shall show the birational equivalence of (\ref{eq: P2 model of Bring}) and (\ref{eq: P4 model of Bring}).

Another plane model of Bring's curve is given by $y^5-(x+1)x^2(x-1)^{-1}=0$, introduced in \cite[Proposition 3.1]{Weber2005} by considering the curve as a cyclic cover of $\mathbb{P}^1$, see \S\ref{sec: quotients by subgroups}. 

%%%%%%%%%%%%%%%%%%%%%%%%%%%%%%%%%%%%%%%%%%%%%%%%%%%%%%%%
%%%%%%%%%%%%%%%%%%%%%%%%%%%%%%%%%%%%%%%%%%%%%%%%%%%%%%%%
\subsection{Basic Properties}
Towards proving the birational equivalence of the two models given for Bring's curve we begin with some basic properties of the curve. We introduce affine coordinates $(x,y)=(X/Z,Y/Z)$ such that the HC-model is $0=f(x,y)=F(x,y,1)$. 
Then $\Res_y(f(x,y),\partial_y f(x,y))=x^4(x^5-1)^2(256x^{10}-837 x^5+3456)$. Investigation of the vanishing of this resultant leads to the following.
\begin{lemma}[\cite{Braden2012}]\label{HCdoublepoints}
 The only singular (double) points in the HC-model of the curve are
$V_k = [\zeta^k : \zeta^{2k} : 1]$ for $k=0, \dots, 4$, where $\zeta=\exp(2\pi i/5)$,  and $V_5 = [1:0:0]$.
\end{lemma}
We shall return these points in due course; each of these singular points desingularizes to two points. We also find
\begin{lemma}\label{lemma: local expansion in Bring's curve}
In the HC-model we have local
expansions about the following points on the curve
\begin{align*}
a&:=[0:0:1]\simeq[2t^3:t:1],& b&:=[0:1:0]\simeq[2t^2:1/t:1],\\
c&:=[1:0:0]_2\simeq[1:t:t^4],& d&:=[1:0:0]_1\simeq[1:t^4:t].
\end{align*}
\end{lemma}
We see from this that
\begin{corollary}
$$
\div(x)=3a+2b-4c-d,\qquad \div(y)=a-b-3c+3d,\qquad R_F = 2a + b+ 3c + \sum_i r_i,
$$ 
where $R_F$ is the ramification divisor corresponding to the map $x:\mathcal{B} \to \mathbb{P}^1$ and the $r_i$ are the roots of the polynomial $256 x^{10} - 837 x^5 + 3456$ appearing in the resultant.
\end{corollary}
\begin{lemma}[\cite{Braden2012}]\label{lemma: basis of differentials}
The genus of Bring's curve is $g=4$ and we have the ordered basis of (unnormalised) differentials
\begin{align*}
  v_1 &= \frac{(y^3-x)\dif x}
    {\partial_y f(x,y)}, &
  v_2 &= \frac{(y^2x-1)\dif x}
    {\partial_y f(x,y)}, &
  v_3 &= \frac{(y-x^2)\dif x}
    {\partial_y f(x,y)}, &
  v_4 &= \frac{y(x^2-y)\dif x}
    {\partial_y f(x,y)}.
\end{align*}
\end{lemma}

These differentials, while independent as holomorphic differentials, do satisfy an algebraic relation.

\begin{lemma}
The canonical embedding of Bring's curve in $\mathbb{P}\sp3=\pbrace{[v_1:v_2:v_3:v_4]}$ lies on the quadric
\begin{equation}\label{quaddif}
\mathcal{Q}:  v_1v_2+v_3v_4=0  . 
\end{equation}
\end{lemma}
While all full-rank quadrics in $\mathbb{P}^3$ are isomorphic to $\mathbb{P}^1 \times \mathbb{P}^1$ \cite{Vakil2010}, the standard proof of which goes via showing every such quadric can be written in some coordinate system as $\mathcal{Q}$, we will subsequently want an explicit form of the isomorphism.
\begin{prop}\label{prop: quadric of canonical embedding isomorphic to PxP}
$\mathcal{Q}\simeq \mathbb{P}\sp1\times\mathbb{P}\sp1$.
\end{prop}
\begin{proof}
Consider the map 
$\mathcal{Q} \to \mathbb{P}^1 \times \mathbb{P}^1$, 
a modification of the Segre embedding, given by
$$
\varphi: \mathbb{P}\sp1\times\mathbb{P}\sp1\rightarrow \mathbb{P}\sp3,\qquad
\varphi([u:v],[z:w])=[uz:vw:vz:-uw]:=[v_1:v_2:v_3:v_4]  .
$$
Then $\varphi([u:v],[z:w])\in \mathcal{Q}$.

The map $\varphi$ is injective for suppose $\varphi([u:v],[z:w])=\varphi([u':v'],[z':w'])$. Consider first that $v_4\ne0$. Then
$$
-\frac{v_1}{v_4}=\frac{z}{w}=\frac{z'}{w'},\qquad -\frac{v_2}{v_4}=\frac{v}{u}=\frac{v'}{u'} , 
$$
and so correspond to the same point in $\mathbb{P}\sp1\times\mathbb{P}\sp1$. Next suppose that $0=-v_4=uw=u'w'$ and without loss of generality that $u=0$. If $u'=0$ we have injectivity, so suppose that $u'\ne0$. Then $w'=0$ and $[0:vw:vz:0]=[u'z':0:v'z':0]$ requiring $z'=0$ and a contradiction. Thus we have that $\varphi$ is injective.

To show the surjectivity of $\varphi$ we may assume without loss of generality that $v_4\ne0$. Then set
$$
-\frac{v_1}{v_4}=\frac{z}{w},\qquad -\frac{v_2}{v_4}=\frac{v}{u}  .
$$
Indeed we can spot the inverse map $\varphi^{-1} : \mathcal{Q} \to \mathbb{P}^1 \times \mathbb{P}^1$ to be given by 
\[
\varphi^{-1}([v_1:v_2:v_3:v_4]) = \left \lbrace \begin{array}{cc}
    ([v_1:v_3], [v_3:v_2]) & \pbrace{v_1, v_3} \neq \pbrace{0} \neq \pbrace{v_2, v_3}, \\
    ([-v_4:v_2], [v_3:v_2]) & \pbrace{v_1, v_3} = \pbrace{0} \neq \pbrace{v_2, v_3} ,\\
    ([v_1:v_3], [v_1:-v_4]) & \pbrace{v_1, v_3} \neq \pbrace{0} = \pbrace{v_2, v_3} ,\\
    ([-v_4:v_2], [v_1:-v_4]) & \pbrace{v_1, v_3} = \pbrace{0} = \pbrace{v_2, v_3}.
\end{array} \right .
\]
\end{proof}

\begin{prop}[\cite{Braden2012}]\label{prop: Canonical divisor of Bring's curve}
The canonical divisor class on Bring's curve is $[\mathcal{K}_{\mathcal{B}}] = [a+2b+3c]$.
\end{prop}
\begin{proof}
This may be shown analytically or (as given in the accompanying notebook) entirely using computer algebra. To show this analytically we use $\div(dx) = -2(4c+d) + R_F$.
\end{proof}

We now establish the main result of this subsection.
\begin{prop}\label{prop: HC valid model for Bring}
The HC-model (\ref{eq: P2 model of Bring}) is a model of Bring's curve (\ref{eq: P4 model of Bring}).
\end{prop}
\begin{proof}
We will prove this by explicitly constructing the birational transform.  To do so we make use of the proof of Theorem 3 in \cite{Dye1995}, where the author considers a particular Clebsch hexagon\footnote{For the purposes of this definition a hexagon is a set of 6 points in $\mathbb{P}^2$, no three of which are collinear, called the vertices. A \bam{Brianchon point} of a hexagon is a non-vertex point through which 3 edges (the lines joining two distinct vertices) pass. A \bam{Clebsch hexagon} is a hexagon with 10 Brianchon points \cite{Dye1991}.}, constructs a pencil of plane sextics from this, and finds Bring's curve as the canonical model of a distinguished point in this pencil. By assuming that the HC-model is already the distinguished curve in a pencil, we can construct the birational map. Note this is a fundamentally different approach to the that originally taken with the HC-model, which was derived from considerations of the modular theory of the curve.

Explicitly, Dye introduces $j$ as a solution to $j^2-j-1=0$ and then defines the pencil of curves\footnote{We are using the coordinates $[x,y,z]$ here, distinct from $[X,Y,Z]$,
to highlight that these are not those of the HC-model.} $S_\lambda = S+ \lambda \left \lvert \mathcal{C} \right \rvert^3$ where 
\begin{equation}\label{dyespencil}
\begin{split}
    S(x,y,z) &= (x+jy)^6+(x-jy)^6+(y+jz)^6+(y-jz)^6+(z+jx)^6+(z-jx)^6  , \\
    \mathcal{C}(x,y,z) &= x^2 + y^2 + z^2  .
\end{split}
\end{equation}
Next Dye considers the Clebsch hexagon $H$ with vertices
\[
(1, \pm j, 0)  , \quad (0, 1, \pm j)  , \quad (\pm j, 0, 1) ,
\]
for which the corresponding 10 Brianchon points are at 
\[
(\pm j^2, 1, 0)  , \quad (0, \pm j^2, 1)  , \quad (1, 0, \pm j^2)  , \quad (1, \pm 1, \pm 1).
\]
Dye shows that there is a unique member of the pencil, which he calls $\Gamma$, that contains the vertices. Moreover, $\Gamma$ has the vertices and only the vertices as double points. 

To get a canonical model for $\Gamma$, which Dye show's has genus 4, we now need a little theory from \cite[pp.~122-124]{Semple1949}, namely that a generic cubic surface in $\mathbb{P}^3$ is birational to a system of plane cubics through 6 base points. We apply this taking these six points to be the vertices of the hexagon $H$. Such a generic cubic surface $F \subset \mathbb{P}^3$ can be written as the vanishing of a determinant 
\[
\begin{vmatrix} u_1 & v_1 & w_1 \\ u_2 & v_2 & w_2 \\ u_3 & v_3 & w_3 \end{vmatrix} = 0  , 
\]
where $u_1, \dots, w_3$ are linear homogeneous functions of the $\mathbb{P}^3$ coordinates, that is to say $P=[L_1 : L_2 : L_3 : L_4]\in F\subset \mathbb{P}^3$ if and only if there exists $P^\prime=[X : Y : Z] \in \mathbb{P}^2$ such that 
\begin{align*}
    X u_i(P) +Y v_i(P) + Z w_i(P) = 0  .
\end{align*}
Thinking of $P^\prime$ as a point in a plane $\Pi$ we get a birational transformation $\Pi \leftrightarrow F$, $P^\prime \leftrightarrow P$. The map $\Psi : \Pi \to F$ will have the $L_a$ as homogeneous cubics in $X, Y, Z$. To see this rewrite the determinant equation as (for $i=1,2,3$)
\[
\sum_{a=1}\sp{4} a_{ia}(P^\prime) L_a = 0
\]
for some $a_{ia}$ linear homogeneous in the $X,Y,Z$. On each affine patch $L_a \neq 0$ solving this involves inverting a $3 \times 3$ matrix whose entries are linear homogeneous polynomials in $X,Y,Z$. Likewise, given the $L_a$, we have a 3-parameter family of cubics given by 
\[
a L_1 + b L_2 + c L_3 + d L_4 = 0 \quad \text{for} \quad [a:b:c:d] \in \mathbb{P}^3  .
\]
A cubic in $\mathbb{P}^2$ has 10 projective coefficients, and so a 3-parameter family is defined by 6 constraints. Generically we can take those constraints to come in the form of intersection with 6 generic points $O_i \in \Pi$. 

For our purposes, the six points we  want to intersect with are the vertices of the Clebsch hexagon, which are the double points of the exceptional curve $\Gamma$. Assuming the the HC-model gives such an exceptional curve, we take the points $V_k$ identified in Lemma \ref{HCdoublepoints}. Hence, if we write a generic cubic in $X, Y, Z$ as 
\[
a_0 X^3 + a_1 X^2 Y + a_2 X^2 Z + a_3X Y^2 + a_4 X Y Z + a_5 X Z^2 + a_6 Y^3 + a_7 Y^2 Z + a_8 Y Z^2 + a_9 Z^3  
\]
the equations on the coefficients we get (coming from intersecting with $V_k$ ($k=0,\dots,4$) and $V_5$ respectively) are
\begin{align*}
    a_0 \zeta^{3k} + a_1 \zeta^{4k} + a_2 \zeta^{2k} + a_3 + a_4 \zeta^{3k} + a_5 \zeta^k + a_6 \zeta^{k} + a_7 \zeta^{4k} + a_8 \zeta^{2k} + a_9 &= 0, \\
    a_0   &= 0.  
\end{align*}
Setting the coefficients of $\zeta^{nk}$ to be zero gives us the 3-parameter family of cubics 
\begin{align*}
 0 &= a X^2 Y + b X^2 Z + cX Y^2  + d X Z^2 -d Y^3 -a Y^2 Z -b Y Z^2 -c Z^3  , \\
 &:= aL_1 + b L_2 + cL_3 + dL_4  .
\end{align*}

Comparing the coefficients we see that our map into $\mathbb{P}^3$ is (essentially) the canonical embedding. as 
\[
[v_1 : v_2 : v_3 : v_4 ] = [-L_4 : L_3 : -L_2 : L_1]  . 
\]
One can check, using for example Gr\"obner bases, that the $L_a$ satisfy the equation
\[
L_2 L_4^2 - L_1^2 L_4 - L_1 L_3^2 + L_2^2 L_3= 0. 
\]
This is the cubic we call $F$. One can use the package \texttt{Cremona} \cite{Stagliano2018} in Macaulay2 to check that the map $\Psi$ is birational. Note one needs to make sure that the range is chosen such that the map is explicitly surjective, not just use the implicit knowledge that the map is surjective onto its image. Doing so and asking for the inverse map gives 
\begin{align*}
    [X:Y:Z] = [L_2^2 - L_1 L_3 : L_1 L_4 : L_2 L_4] . 
\end{align*}
For example, we can see 
\begin{align*}
\frac{L_2^2 - L_1 L_3}{L_1 L_4} &= \frac{\pround{X^2 Z - Y Z^2}^2 - \pround{X^2 Y - Y^2 Z}\pround{X Y^2 - Z^3}}{\pround{X^2 Y - Y^2 Z} \pround{X Z^2 - Y^3}}  ,  \\
&= \frac{X^4 Z^2 -X^2 Y Z^3 -X^3 Y^3 + X Y^4 Z}{X^3 Y Z^2 - X Y^2 Z^3 - X^2 Y^4 + Y^5 Z} = \frac{X}{Y}  ,
\end{align*}
and
\[
\frac{L_1L_4}{L_2 L_4} = \frac{L_1}{L_2} = \frac{X^2 Y - Y^2 Z}{X^2 Z - Y Z^2 } = \frac{(X^2 Z - Y Z^2)(Y/Z)}{X^2 Z - Y Z^2 } = \frac{Y}{ Z}  . 
\]
We now have a cubic surface in $\mathbb{P}^3$ corresponding to the system of curves intersecting the $V_i$. From \cite[pp.~198-201]{Hirschfeld1986} we know there exists a coordinate system in which this cubic can be written as the subset of $\mathbb{P}^4$ 
\[
H_3 = 0 = H_1  
\]
(this is called `the' Clebsch surface). To look for such a coordinate system, we write this surface in $\mathbb{P}^3$ as 
\begin{align*}
0 &= x_1^{2} x_2 + x_1 x_2^{2} + x_1^{2} x_3 +  x_2^{2} x_3 + x_1 x_3^{2} + x_2 x_3^{2} + x_1^{2} x_4 +  x_2^{2} x_4  +  x_3^{2} x_4 + x_1 x_4^{2} + x_2 x_4^{2} + x_3 x_4^{2} \\
&\phantom{=} + 2 x_1 x_2 x_3 + 2 x_1 x_2 x_4 + 2 x_1 x_3 x_4 + 2 x_2 x_3 x_4   .
\end{align*}
From here, one can use the fact that the 10 Eckardt points\footnote{An \bam{Eckardt point} of a surface is a point where three lines contained within the surface intersect \cite{Hirschfeld1986}.} of the cubic in $L_a$ form the 10 vertices of a pentahedron \cite{Hirschfeld1986}, and that in the coordinate system of the Clebsch surface in $\mathbb{P}^4$ these lie at the permutations of $[1:-1:0:0:0]$ \cite{Edge1978}. This gives us a possible way a spotting the transform if we can calculate the Eckardt points of $F$. To do this, we use the identification from \cite{Dye1995}, that the 3 lines in $F$ intersecting to give an Eckardt point come from the 3 edges of $H$ intersecting at a Brianchon point. To this end we find the images of the $V_i V_j$ for which we give a generating set of the ideal corresponding to the line, for example 
\begin{enumerate}
    \item $\Psi(V_0 V_5) : \pangle{L_4 - L_3, L_2 - L_1}$, 
    \item $\Psi(V_1 V_2) : \pangle{L_4 + (\zeta^2 + \zeta+1 )L_2 + (\zeta^4 - \zeta^2 - \zeta)L_1,  L_3 + (\zeta^3+ 2 \zeta^2 + \zeta)L_2 + (\zeta^3 + \zeta^2 + \zeta) L_1}$,
    \item $\Psi(V_3 V_4) : \pangle{L_4 + (-\zeta^2 - \zeta)L_2 + (\zeta^3 + 2\zeta^2 + \zeta)L_1,  L_3 + (\zeta^3 - \zeta - 1)L_2 + (\zeta^2 + \zeta + 1)L_1}$.
\end{enumerate}
Then
\[
\Psi(V_0 V_5) \cap \Psi(V_1 V_2) \cap \Psi(V_3 V_4) = [-\zeta^3 - \zeta^2 - 1 : -\zeta^3 - \zeta^2 - 1 : 1 : 1] = \Psi(V_0 V_5 \cap V_1 V_2 \cap V_3 V_4)  . 
\]
One can do likewise to find the other Eckardt points. 

Armed now with the knowledge of the Eckardt points, we can find appropriate projective transforms $A$ that biject the sets of Eckardt points by acting $(x_1, x_2, x_3, x_4)^T = A(L_1, L_2, L_3, L_4)^T$, for example 
\[
A = \begin{pmatrix}
\zeta^3 & -1 & -\zeta^{2} & \zeta \\ 1 & -\zeta^3 & -\zeta & \zeta^{2} \\ \zeta^{2} & -\zeta & -1 & \zeta^3 \\ \zeta & -\zeta^2 & -\zeta^3 & 1 \end{pmatrix} . 
\]
We can then use again Macaulay2 to check that this transform then gives us the correct cubic in $\mathbb{P}^3$. Moreover, if we now restrict to the HC-model in the system of curves that intersect the $V_i$, that is we impose the condition that $F(X, Y, Z)=0$, we get the final degree-2 polynomial on the $x_i$ ($H_2=0$). This quadric is in fact the Schur quadric corresponding to a distinguished double-six\footnote{A \textbf{double six} is a collection of 12 lines $a_1, \dots, a_6, b_1, \dots, b_6$ in $\mathbb{P}^3$, arranged as 
\[
\begin{array}{cccccc}
    a_1 & a_2 & a_3 & a_4 & a_5 & a_6  \\
    b_1 & b_2 & b_3 & b_4 & b_5 & b_6 
\end{array}
\]
so that each line is disjoint from those in the same row and column, but intersects the other 5 lines.} of lines that necessarily exist on a cubic surface constructed as above \cite{Dye1995}. An equivalent rational map is given in \cite[p. 557]{Dolgachev2012}, but the inverse is not provided.

To see this working, let's consider $c$ and $d$, and see that these are desingualrised on the smooth canonical embedding. Taking $[X:Y:Z] = [1:t:t^4]$ we get 
\[
[L_1: L_2: L_3: L_4] = [1: t^3 :t+ t^6 : -t^2]  ,
\]
and so taking the limit we get $[1:0:0]_2 \mapsto [1:0:0:0]$ in $L$ coordinates. Acting with $A$, we get the point in $\mathbb{P}^4$ given by 
\[
[x_1 : x_2 : x_3 : x_4 : x_5] = [\zeta^3 : 1 : \zeta^2 : \zeta : \zeta^4]  . 
\]
Repeating the process taking $[X:Y:Z] = [1:t^4:t]$ gives $[1:0:0]_1 \mapsto [0:1:0:0]$ in $L$ coordinates, which is equivalently 
\[
[x_1 : x_2 : x_3 : x_4 : x_5] = [1 : \zeta^3 : \zeta : \zeta^2 : \zeta^4]  . 
\]
This makes sense, as the change $Y \leftrightarrow Z$ corresponds to $L_1 \leftrightarrow L_2$, $L_3 \leftrightarrow L_4$. 

Moreover, to clarify our last point on the imposition of $H_2=0$, we consider the point $[X:Y:Z] = [0:1:1]$ which does not lie on the HC-model of the curve. Under our birational map this corresponds to the point 
\begin{align*}
    [L_1 : L_2 : L_3 : L_4] &= [1:  1 :1 :1]  
    \Rightarrow [x_1 : x_2 : x_3 : x_4 : x_5] &=  [2-\sqrt{5} : -2+\sqrt{5} : -1 : 1 : 0]  .
\end{align*}
This point does lie on the Clebsch diagonal surface given by $H_1 = 0 = H_3$, but does not satisfy $H_2=0$. 
\end{proof}

Note that in the above proof of the equivalence of the Riemann surfaces, the birational map we constructed was defined over $\mathbb{Q}[\zeta]$. As such, to equate the two algebraic curves we need to be working over a field containing $\mathbb{Q}[\zeta]$. We will later see when looking at quotients of Bring's curve that it is insufficient to work over $\mathbb{Q}$. Indeed, note that over $\mathbb{Q}$ there are no solutions to the equations defining Bring's curve in the $\mathbb{P}^4$-model.  
%%%%%%%%%%%%%%%%%%%%%%%%%%%%%%%%%%%%%%%%%%%%%%%%%%%%%%%%
%%%%%%%%%%%%%%%%%%%%%%%%%%%%%%%%%%%%%%%%%%%%%%%%%%%%%%%%
\subsection{The Period Matrix}\label{subsec: period matrix}
We record now a result from the literature.
\begin{theorem}[\cite{Riera1992}, \cite{Gonzalez2000}]
Define the matrices $M$, $M_S$ by
\[
M =  \begin{pmatrix}
4 & 1 & -1 & 1 \\ 1 & 4 & 1 & -1 \\ -1 & 1 & 4 & 1 \\ 1 & -1 & 1 & 4
\end{pmatrix}, \quad
M_S =  \begin{pmatrix}
4 & -1 & -1 & -1 \\ -1 & 4 & -1 & -1 \\ -1 & -1 & 4 & -1 \\ -1 & -1 & -1 & 4
\end{pmatrix}.
\]
A Riemann matrix for Bring's curve is given by
$\tau_{\mathcal{B}} = \tau_0 M$, where $\tau_0 = -0.5+0.186676i \, (6.d.p)$ is given by the conditions 
\[
j(\tau_0) = - \frac{29^3 \times 5}{2^5}, \quad j(5\tau_0) = -\frac{25}{2} .
\]
Further, there exists a symplectic transformation such that $\tau_{\mathcal{B}} = \tau_0 M_S$.
\end{theorem}
\begin{proof}
This was first shown in \cite{Riera1992}, but the equations for $j(\tau_0)$, $j(5\tau_0)$ were incorrectly swapped as first noted\footnote{Note the approximation of $\tau_0$ in \cite{Braden2012} contains a typographic error.} in \cite{Braden2012}, where the above period matrix was also calculated via the HC-model. 

Moreover, in \cite{Weber2005}, a period matrix is constructed with respect to a homology basis with intersection matrix 
\[
I_W = \begin{pmatrix}
0 & 1 & 1 & -1 & -1 & 0 & 0 & 1 \\ -1 & 0 & 1 & 1 & 0 & 1 & 0 & -1 \\ -1 & -1 & 0 & 1 & 0 & -1 & 0 & 0 \\ 1 & -1 & -1 & 0 & 0 & 0 & 1 & 0 \\ 1 & 0 & 0 & 0 & 0 & 1 & 1 & -1 \\ 0 & -1 & 1 & 0 & -1 & 0 & 1 & 1 \\0 & 0 & 0 & -1 & -1 & -1 & 0 & 1 \\ -1 & 1 & 0 & 0 & 1 & -1 & -1 & 0 
\end{pmatrix}  , 
\]
that is given as $\Omega = (A, B)$ where the columns $A_k, B_k$ are 
\[
A_k = \begin{pmatrix} \zeta^{k} (1-\zeta^2) \\ \zeta^{2k+3/2}(1-\zeta^4)(l\Phi-1) \\ \zeta^{4k+3}(1-\zeta^3)\Phi(1-l) \\ \zeta^{3k+2}(1-\zeta)l \end{pmatrix}, \quad B_k = \begin{pmatrix}  \zeta^{2k+3/2}(1-\zeta^4)(l\Phi-1) \\ \zeta^{4k+3}(1-\zeta^3)\Phi(1-l) \\ \zeta^{3k+2}(1-\zeta)l \\ \zeta^{k} (1-\zeta^2) \end{pmatrix}  ,
\]
for $k=0, \dots, 3$, $\Phi = \frac{1}{2}(1+\sqrt{5})$, and $l = \left \lvert {\frac{I(-1,0)}{I(-\infty, -1)}} \right \rvert \approx 0.848641$ where 
\[
I(a,b) = \int_a^b (t-1)^{-1/5} t^{-3/5} (t+1)^{-4/5}  dt  . 
\]
Note \cite[Lemma 5.1]{Weber2005} uses $\zeta= \exp(2 \pi i/10)$, whereas we take $\zeta= \exp(2 \pi i/5)$. As such, the columns of $\Omega$ look different to those of Weber in terms of the exponent of $\zeta$, but they do indeed agree.

One can find (using Sage) the matrix 
\[
C = \begin{pmatrix}
0 & 1 & -2 & 1 & 1 & 1 & 0 & -1 \\ 0 & 1 & -1 & 1 & 0 & 0 & -1 & 0 \\ 0 & 0 & -1 & 0 & 0 & 1 & 1 & 0 \\ 1 & 1 & -2 & 2 & 0 & 1 & -1 & -1 \\ 0 & 0 & 0 & 0 & 0 & 0 & 1 & 1 \\ 0 & 0 & 1 & 0 & 0 & 0 & 0 & 0 \\ 0 & 0 & 0 & 0 & 0 & 0 & 0 & 1 \\ 0 & 0 & 0 & 1 & 0 & 0 & 0 & 0
\end{pmatrix}
\]
such that $C^T I_W C = J_g = \begin{pmatrix} 0 & \id_g \\ -\id_g & 0 \end{pmatrix}$, i.e. $C$ `canonicalises' the homology basis. This means we get a Riemann matrix $\tau_W = (AC)^{-1} (BC) = C^{-1}A^{-1} B C$, and one can numerically find that the matrix 
\[
R = \begin{pmatrix} \delta & \beta \\ \gamma & \alpha \end{pmatrix} = \begin{pmatrix}
0 & 0 & 0 & 1 & 0 & 0 & 1 & 2 \\ 0 & -1 & 1 & 0 & -1 & -2 & 1 & 1 \\ 1 & 0 & 1 & 0 & 2 & 1 & 1 & 1 \\ 0 & 1 & 0 & 0 & 1 & 2 & 0 & 0 \\ 1 & 0 & 0 & 0 & 2 & 1 & 0 & 1 \\ 1 & 0 & 1 & -1 & 1 & 1 & 1 & -1 \\ -1 & 0 & 0 & 1 & -1 & -1 & 1 & 2 \\ 1 & 0 & 1 & -1 & 1 & 2 & 1 & -1
\end{pmatrix}
\]
relates $\tau_W$ to $\tau_{\mathcal{B}}$ by 
\begin{equation}\label{transRRW}
   \tau_{\mathcal{B}} = (\delta + \tau_W \gamma)^{-1}(\beta+ \tau_W \alpha) .
\end{equation}
Here $R$ is a symplectic transform with respect the standard symplectic form $J_g$. 

For a fourth proof, we may consider the Riemann matrix numerically calculated by SageMath, and find a transform to $\tau_{\mathcal{B}}$, as is done in the corresponding notebook. 
\end{proof}

Riera and Rodriguez (hereafter abbreviated to R\&R) constructed the constraints on $\tau_0$ via $j$-invariants by considering the quotients by group actions of the curve to elliptic curves \cite{Riera1992}. They show that these constraints give a unique value of $\tau_0$ modulo $\Gamma_0(5)$, or equivalently in the language of \cite{Gonzalez2000} that $\tau_0$ gives a distinguished point in the modular curve $X_0(5)$. As we will see later in \S\ref{subsec: quotient by 3-cycle}, there are additional quotients to elliptic curves not considered by  R\&R, which have $j$-invariants 
\begin{align}
    j(15 \tau_0) &= -\frac{5^2 \times 241^3}{2^3}  , \\
    j(3 \tau_0) &= \frac{5 \times 211^3}{2^{15}}   .
\end{align}
The latter is identified in relation to Bring's curve in \cite[Exercise 8.3.2c]{Serre2008}. Serre says that this curve (50H) and 50E (using the naming convention of \cite{Birch1975}) with $j(5\tau_0)$ are 15-isogenous over $\mathbb{Q}$. This isogeny is not too mysterious when we think on the level of the corresponding elliptic curves over $\mathbb{C}$ as $\mathbb{C}/\pangle{1,3 \tau_0}$ and $\mathbb{C}/\pangle{1,5 \tau_0}$, wherein the isogeny $\mathbb{C}/\pangle{1,3 \tau_0} \to \mathbb{C}/\pangle{1,5 \tau_0}$ is the composition of the quotients by the maps $z \mapsto z+\tau_0$ and $z \mapsto z+ 1/5$ respectively. There is a complete $\mathbb{Q}$-isogeny class of elliptic curves of order four with periods $\tau_0, 3 \tau_0, 5 \tau_0$ and $15 \tau_0$ \cite{Birch1975}. 

It follows from the above that $\tau_0$ is transcendental. By a theorem of Schneider \cite{Baker2008}, if $j(z)$ is rational, $z$ is either transcendental or an element of a quadratic imaginary field with class number 1. However, the $j$-invariants of elements of these fields are known to be integers, which doesn't occur in the case of $\tau_0$. In Weber's form of the period matrix, the transcendentality comes about because of the constant $l$, which is a ratio of Schwarz-Christoffel integrals which arise from the map of a Euclidean quadrilateral to a hyperbolic quadrilateral. 
%%%%%%%%%%%%%%%%%%%%%%%%%%%%%%%%%%%%%%%%%%%%%%%%%%%%%%%%
%%%%%%%%%%%%%%%%%%%%%%%%%%%%%%%%%%%%%%%%%%%%%%%%%%%%%%%%
\subsection{The Automorphism Group}\label{subsec: automorphism group}

As previously noted, the $\mathbb{P}^4$-model of Bring's curve shows that the automorphism group $\Aut(\mathcal{B})$ contains $S_5$. Hurwitz's theorem on the automorphism group $G=\Aut(\mathcal{C})$ of a curve of genus $g\ge2$ yields that
$|G|\in\{84(g-1), 48(g-1), 40(g-1),\ldots \}$ and that if $p$ is a prime where $p\big| |G|$, then $p\in\{2,\ldots,g, g+1,2g+1\}$.
%; further the order of any element of $G$ is less than or equal to $2(2g+1)$.
For genus $4$ this means $p\in\{2,3,5\}$ and so the maximal automorphism group of a genus $4$ curve has order strictly less than $84(g-1)$. In 1895 Wiman showed\footnote{An English translation of Wiman's paper has been produced \cite{DisneyHogg2022}.}

\begin{prop}[\cite{Wiman1895}]
$\Aut(\mathcal{B}) = S_5$. This is the maximal possible automorphism group for a genus-4 surface, and Bring's curve is the only curve to achieve it.
\end{prop}

Wiman's proof was enumerative and produced equations for the curves with a given automorphism group. He had previously dealt with the hyperelliptic curves and their automorphism groups and for non-hyperelliptic curves $\mathcal{C}$ of genus 4 he began with the canonical embedding of the curve in $\mathbb{P}^3$ which meant that their automorphisms could be taken to be collineations, and so forming a subgroup of $PGL_4(\mathbb{C})$. Wiman's approach then used the fact (see \cite[IV.5.2.2]{Hartshorne1977}) that the canonical embedding of a non-hyperelliptic curve of genus 4 is the complete intersection of an irreducible cubic surface and a unique quadric surface (which is either irreducible, or a cone). The uniqueness of the quadric $Q$ means that automorphisms of the curve become automorphisms of $Q$. In the case where $Q$ is of full rank we have that it is isomorphic to $\mathbb{P}^1 \times \mathbb{P}^1$, hence we know $\Aut(\mathcal{C})$ is isomorphic to a finite subgroup of $\Aut(\mathbb{P}^1 \times \mathbb{P}^1) = C_2 \ltimes (PGL_2(\mathbb{C}) \times PGL_2(\mathbb{C}))$ and that using the coordinates of $\mathbb{P}^1 \times \mathbb{P}^1$ Wiman could express the curve $\mathcal{C}$ as an equation of bidegree $(3,3)$ in terms of these. In the singular case $Q$ is a quadric cone and by projecting from a point of the cone meant that $\mathcal{C}$ could be expressed as a plane sextic. In both cases Wiman could write equations for possible genus $4$ curves and his strategy was to look at the restrictions on these imposed by symmetries of orders $2$, $3$ and $5$. For curves without an order $5$ element the maximal order of symmetry group was $72$ for $Q$ either a smooth quadric or cone. Including an order $5$ element yielded a maximal symmetry group of order $120$ only in the case of the smooth quadric with the resulting curve being Bring's. In the case of Bring's curve the quadric $Q$ is the quadric $\mathcal{Q}$ we determined earlier. Different proofs of Wiman's result may be found in \cite{Kuribayashi1990,Magaard2002}.

It is clear from the $\mathbb{P}^4$-model that $\Aut(\mathcal{B})$ may be realised as projective transforms via the permutation representation of $S_5$ acting on the subspace $\sum_i x_i = 0$; it is also clear that it may be  be realised as a subgroup of $PGL_4(\mathbb{C})$ via the induced action on the differentials or equivalently the $L_a$'s. What is non-trivial is the following fact.
\begin{theorem}[\cite{Dye1995}]\label{autA5}
The $A_5$ subgroup of $\Aut(\mathcal{B})$ can be realised as a group of collineations in the HC-model, that is, can be realised as a subgroup of $PGL_3(\mathbb{C})$.
\end{theorem}
\begin{proof}
We will be explicit about the construction here as this will be profitable later; \cite{Braden2012} explains how this representation follows from \cite{Dye1995}. The group $A_5$ has two inequivalent irreducible three-dimensional representations, one of which is given by $\pangle{R, S}$ where 
$$
R:=\frac1{\sqrt5}\begin{pmatrix}
    1 & 2                 & 2 \\
    1 & \zeta^2+\zeta^{-2}&\zeta+\zeta^{-1}  \\
    1 & \zeta+\zeta^{-1}  & \zeta^2+\zeta^{-2}
  \end{pmatrix}, \, {R}^2=I , \,  
  S:=\begin{pmatrix}1&&\\ &\zeta&\\ &&\zeta^{-1}\end{pmatrix},\  {S}^5=I, \,  (RS)^3=I  .
$$
Note the other inequivalent irreducibel 3-dimensional representation comes from replacing $\zeta$ in the above with $\zeta^2 $ or $\zeta^3$. The invariants of the representation $\pangle{R, S}$ (when acting on $(X,Y,Z)^T$ via left multiplication) are
\begin{align*}
i_2 &=\pround{\frac{X}{2}}^2+\sum_{k=0}\sp4\left(\frac{\frac{X}{2}+Y\zeta^k+Z\zeta^{-k}}{\sqrt{5}}\right)^2=\frac{1}{2}X^2+2YZ,\\
i_6 &=\pround{\frac{X}{2}}^6+\sum_{k=0}\sp4\left(\frac{\frac{X}{2}+Y\zeta^k+Z\zeta^{-k}}{\sqrt{5}}\right)^6,\\
i_{10} &=\pround{\frac{X}{2}}^{10}+\sum_{k=0}\sp4\left(\frac{\frac{X}{2}+Y\zeta^k+Z\zeta^{-k}}{\sqrt{5}}\right)^{10},\\
i_{15} &= \left \lvert\frac{\partial\{ i_2,i_{6},i_{10}\} }{\partial \{X,Y,Z\}} \right \rvert.
\end{align*}
There is a polynomial relation between $i_{15}\sp{ 2}$ and $i_2$, $i_6$, $i_{10}$. In particular the vanishing of $i_6-\lambda i_2^3$ gives us Dye's one-parameter family of $A_5$-invariant sextics in $\mathbb{P}\sp2$; this pencil appears to have first been studied by Winger \cite{Winger1925}. This pencil yields genus-10 curves for generic $\lambda$ \cite{Dye1995}.\footnote{A variation on this pencil has been used to explain why Bring's curve is uniquely defined as an $A_5$-invariant curve of genus 4, but there is a 1-parameter family of dimension-4 $A_5$-invariant principally polarised abelian varieties, deforming $\Jac\mathcal{B}$ \cite{Zi2021, Looijenga2021}. The paper \cite{Melliez2003} gives further interesting representation theoretic perspectives on Bring's curve.} The special value of $\lambda=13/100$\footnote{For $\lambda=17/180$ the curve is of genus 0; for $\lambda=1/10$ the curve is reducible.} yields a genus-4 curve, namely Bring's curve, as we have that
$$
\frac1{12}\left(100 i_6-13 i_2\sp{ 3} \right)= X(Y^5+Z^5)+(XYZ)^2-X^4 YZ-2(YZ)^3 = F(X,Y,Z)  .
$$
\end{proof}

To complete our picture, we use the following result. 
\begin{prop}\label{prop: U is an automorphism}
The map 
\[
U:\pround{x,y} \mapsto \pround{-\frac{y^5 + x^3y - 3xy^2 + 1}{(y-x^2)(y^3-x)}, - \frac{y^2x-1}{y^3-x}} 
\]
is an automorphism of the HC-model, and together with $R$ and $S$ generates the entire automorphism group $S_5$. 
\end{prop}
\begin{proof}
The proof that $U$ is an automorphism is simple algebraic verification. In order to find this map, we adapted the methods of \cite{Bruin2019}. To see that $\pangle{R,S,U} \cong S_5$, note that $U$ is of order $4$, for example by checking that it has the orbit $a \mapsto c \mapsto b \mapsto d \mapsto a$. As such, $U$ corresponds to an odd permutation under the isomorphism $\Aut(\mathcal{B}) \cong S_5$, and a single odd permutation and all of $A_5$ together generate $S_5$.  
\end{proof}
By fixing a map from the HC-model to the $\mathbb{P}^4$-model as we did in the proof of Proposition \ref{prop: HC valid model for Bring}, we have fixed an isomorphism from the automorphism group of the curve (in the HC-model) to $S_5$, which we shall denote $\psi : \pangle{R,S,U} \to S_5$. For example, it is simple to verify that we have $U^2([X:Y:Z]) = [X:Z:Y]$. We see $U^2([L_1: L_2: L_3: L_4]) = [L_2 : L_1 : L_4 : L_3]$, and so $U^2([x_1: x_2 : x_3 : x_4 : x_5]) = [x_2 : x_1 : x_4 : x_3 : x_5]$, that is $(12)(34)$. Through similar calculation we can find 
\begin{align*}
    \psi(R) &= (13)(24), \quad
    \psi(S) = (1 3 4 2 5),\quad
    \psi(U) = (1 3 2 4).
\end{align*}There are myriad choices that can be made in constructing the birational transformation (such as the labelling of the coordinates in $\mathbb{P}^3$ and the ordering of the rows of $A$, or indeed composing with any automorphism of the $\mathbb{P}^4$-model), and changing these would give different isomorphisms to $S_5$. 

As we have previously described, the uniqueness of the
quadric $Q$ whose intersection with a cubic yields the canonical model of the curve leads to an isomorphism of the automorphism group of the curve and a subgroup of $\Aut(Q)$. For Bring's curve $\mathcal{Q} \cong\mathbb{P}^1 \times \mathbb{P}^1$ and we obtain an  isomorphism from $S_5$ to a subgroup of $\Aut(\mathbb{P}^1 \times \mathbb{P}^1) = \mathbb{Z}_2 \ltimes (PGL_2(\mathbb{C}) \times PGL_2(\mathbb{C}))$ which we now write down. Let $([u:v], [z:w])$ be the coordinates on $\mathbb{P}^1 \times \mathbb{P}^1$, then using the birational map constructed in Proposition \ref{prop: HC valid model for Bring} one can conjugate the standard irreducible 4-dimensional representation of $S_5$ on $[x_1: x_2 : x_3 : x_4]$ to an action on $[v_1 : v_2 : v_3 : v_4]$. The resulting action of $(12)$ is (projectively)
\begin{align}\label{can12}
    \begin{pmatrix}
    v \\ u
    \end{pmatrix} &\mapsto \begin{pmatrix}
    j & -1 \\ -1 & -j
    \end{pmatrix} \begin{pmatrix}
    w \\ z
    \end{pmatrix}:=A \begin{pmatrix}
    w \\ z
    \end{pmatrix}, \quad
    \begin{pmatrix}
    w \\ z
    \end{pmatrix} \mapsto \begin{pmatrix}
    -1 & j-1 \\ j-1 & 1
    \end{pmatrix}\begin{pmatrix}
    v \\ u
    \end{pmatrix}:=B\begin{pmatrix}
    v \\ u
    \end{pmatrix}, 
\end{align}
where $j = - \zeta^3 - \zeta^2$ satisfies $j^2-j-1=0$; that is it is the $j$ defined by Dye. One can show that $(34)$ has the same action, where the other root $j^\prime = \zeta^3+\zeta^2+1$ is taken. One can check that $AB = 1 \in PGL_2(\mathbb{C})$, consistent with the fact that $(12)^2=1 \in S_5$. Note the transposition interchanges the two copies of $\mathbb{P}^1$, which is the action of the semi-direct product with $\mathbb{Z}_2$. Combining the two transforms one gets that $(12)(34)$ acts as 
\[
[u:v] \mapsto [-v:u] \, , \quad [z:w] \mapsto [-w:z]  . 
\]
This fixes each copy of $\mathbb{P}^1$ and is the antipodal map on each. 
Moreover, we can calculate the action of $(145)$ on the copies to be 
\begin{align}\label{can145}
    \begin{pmatrix}
    v \\ u
    \end{pmatrix} &\mapsto \begin{pmatrix}
    \zeta & -\zeta^3-\zeta \\ -\zeta^2-1 & -\zeta^2
    \end{pmatrix} \begin{pmatrix}
    v \\ u
    \end{pmatrix}, \quad
    \begin{pmatrix}
    w \\ z
    \end{pmatrix} \mapsto \begin{pmatrix}
    \zeta^3+1 & \zeta^2 \\ \zeta^4 & -\zeta^3-\zeta
    \end{pmatrix}\begin{pmatrix}
    w \\ z
    \end{pmatrix} . 
\end{align}
As $\pangle{(12)(34), (145)} \cong A_5$, we discover that the action of $A_5$ does not interchange the two copies of $\mathbb{P}^1$, but odd-parity elements in $S_5$ do. Here $A_5$ is given by the diagonal embedding in $PGL_2(\mathbb{C}) \times PGL_2( \mathbb{C})$. 
%%%%%%%%%%%%%%%%%%%%%%%%%%%%%%%%%%%%%%%%%%%%%%%%%%%%%%%%
%%%%%%%%%%%%%%%%%%%%%%%%%%%%%%%%%%%%%%%%%%%%%%%%%%%%%%%%
%%%%%%%%%%%%%%%%%%%%%%%%%%%%%%%%%%%%%%%%%%%%%%%%%%%%%%%%
%%%%%%%%%%%%%%%%%%%%%%%%%%%%%%%%%%%%%%%%%%%%%%%%%%%%%%%%
\section{Geometric Points}\label{sec: geometric points}
We now have a good understanding of how the automorphism group acts on the curve, and so before looking at quotient Riemann surfaces in \S\ref{sec: quotients by subgroups} we want to first consider orbits of points that have geometric significance on the curve. These points will have important connections to the function theory of the curve; they are also related to physical aspects of Euclidean realisations (i.e. can be immersed in Euclidean 3-space) of the curve. Such orbits are characterised by the following result from Wiman.
\begin{prop}[\cite{Wiman1895}]
There are only 3 orbits orbits of points of size less than $120$ on $\mathcal{B}$ and these have sizes 24, 30, and 60 respectively.
\end{prop}
\begin{proof}
By \cite[III.7.7]{Farkas1992} we know the stabilizer of a point must be a cyclic group. The cyclic subgroups of $S_5$ are $C_2$, $C_3$, $C_4$, $C_5$, $C_6$; the corresponding orbits would thus be of (respective) sizes $60$, $40$, $30$, $24$ or $20$. To obtain Wiman's result we must show that $C_3$ does not fix a point on Bring's curve. Using Riemann-Hurwitz we have
$$
3=120(g-1)+30 a_{60}+40 a_{40}+45 a_{30}+48 a_{24}
+50 a_{20}
$$
where $a_k\ge0$ are the number of $S_5$ orbits of size
$k$. There are no solutions to this for $g\ge1$. For $g=0$ we have the unique solution $1=a_{60}=a_{30}=a_{24}$. This shows there can be no points with $C_3$ stabilizer.
\end{proof}
These points and corresponding geometric structures are important when relating Clebsch's diagonal surface to Hilbert modular surfaces \cite{Hirzebruch1976, Bell2004}. We identify these orbits as the geometric points on the curve defined in \cite{Singerman1997}.\footnote{Being a geometric point on a curve is a priori not an interesting statement unless we know the corresponding map is regular, as we have in this case.} Explicitly they are the vertices, face-centres, and edge-centres of the the universal map $\pbrace{5,4}_6$ - the Petrie polygon (as defined in \cite[\S 8.6]{Coxeter1980}) of degree 6 coming from the tiling of the hyperbolic disk by pentagons, where 4 meet at a vertex \cite{Singerman1988}. It is noted in \cite{Weber2005} that this tessellation has a Euclidean realisation as a dodecadodecahedron (Figure \ref{fig: DID}). This has 30 vertices, 60 edges, and 24 faces, giving genus 
\[
g = 1 - \frac{V-E+F}{2} = 4  ,
\]
as we expect. In a recent paper, this connection to the dodecadodecahedron was used to identify Bring's curve as the moduli space of equilateral plane pentagons up to the action of the conformal group \cite{Ramshaw2022}. The (small) stellated dodecahedron (Figure \ref{fig: SSD}) also has genus 4 (having $V=12=F$, $E=30$), coming from the tessellation $\pbrace{5/2, 5} \cong \pbrace{5, 5  |  3}$, which can be interpreted as adding 3 `holes' to the $\pbrace{5,5}$ tessellation \cite[\S 8.5]{Coxeter1980}.  This $\pbrace{5,5|3}$ tessellation has automorphism group $C_2 \times A_5$, which is an index-2 subgroup of $C_2 \times S_5$, the automorphism group of $\pbrace{5,4}_6$. This is due to the map $D_1$ defined in \cite[\S 3.1]{Hendriks2013}, which maps the dodecadodecahedron to the small sellated dodecahedron. We include both these tessellations in Figure \ref{fig: tessellations} below. Klein connects the small stellated dodecahedron to Bring's curve through a degree-3 covering $\mathcal{B} \to \mathbb{P}^1$ constructed from the hyperbolic triangles giving the tessellation \cite{Weber2005}, and we will see this map later in \S\ref{sec: theta characteristics} in a different context.

\begin{figure}
    \centering
     \begin{subfigure}[c]{0.49\textwidth}
         \centering
         \includegraphics[width=\textwidth]{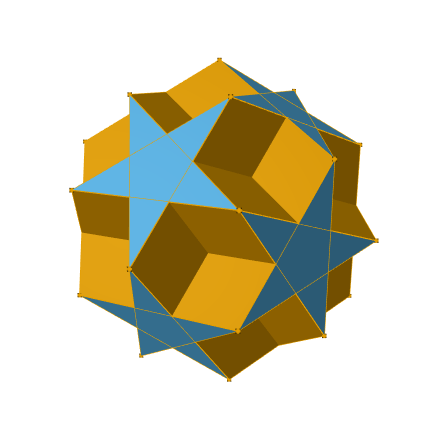}
         \caption{Dodecadodecahedron.}
         \label{fig: DID}
     \end{subfigure}
     \begin{subfigure}[c]{0.49\textwidth}
         \centering
         \includegraphics[width=\textwidth]{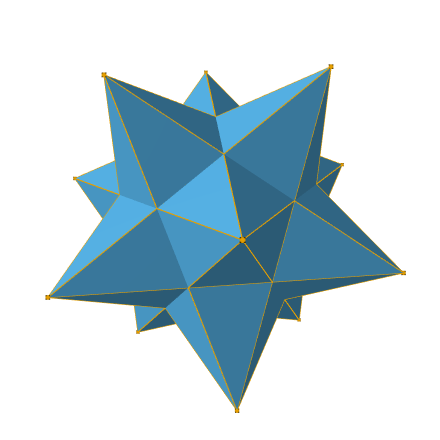}
         \caption{Small Stellated Dodecahedron.}
         \label{fig: SSD}
     \end{subfigure}
     \caption{Geometric realisations.}
     \label{fig: geometric realisations}
\end{figure}

\begin{figure}
    \centering
     \begin{subfigure}[c]{0.49\textwidth}
         \centering
         \includegraphics[width=\textwidth]{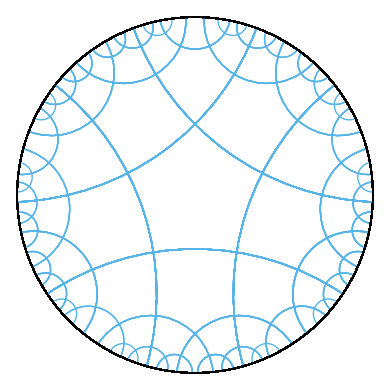}
         \caption{$\pbrace{5, 4}$ tessellation.}
         \label{fig: 54 tessellation}
     \end{subfigure}
     \begin{subfigure}[c]{0.49\textwidth}
         \centering
         \includegraphics[width=\textwidth]{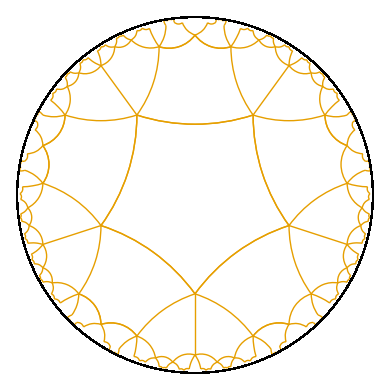}
         \caption{$\pbrace{5, 5}$ tessellation.}
         \label{fig: 55 tessellation}
     \end{subfigure}
     \caption{Hyperbolic tilings. }
     \label{fig: tessellations}
\end{figure}

%%%%%%%%%%%%%%%%%%%%%%%%%%%%%%%%%%%%%%%%%%%%%%%%%%%%%%%%
%%%%%%%%%%%%%%%%%%%%%%%%%%%%%%%%%%%%%%%%%%%%%%%%%%%%%%%%
\subsection{Weierstrass Points}\label{subsec: Weierstrass points}
We recall the definition of an important class of points.
\begin{definition}[\cite{McMullen2014}, Proposition 12.6, Theorem 12.7]\label{def: Weierstrass points}
A \bam{Weierstrass Point} on a curve $X$ of genus $g$ with canonical embedding $X_{c} \subseteq \mathbb{P}^{g-1}$ is a point $P$ satisfying one of the following equivalent conditions:
\begin{enumerate}
    \item there exists a hyperplane $H \subset \mathbb{P}^{g-1}$such that $H \cap X_c$ has multiplicity at least $g$ at $P$,
    \item there exists a holomorphic differential on $X$ vanishing at $P$ with order $\geq g$, 
    \item there exists a meromorphic function on $X$ with poles just at $P$ of order $\leq g$, and 
    \item the Wronskian determinant given by $\operatorname{Wr}(P) = \det \pround{\frac{d^i \omega_j}{dz^i}}_{i,j = 0, \dots, g-1}$, where $\{ \omega_j \}$ is a basis of holomorphic differentials and $z$ is a local coordinate around $P$, vanishes.
\end{enumerate}
\end{definition}
In \cite{Singerman1997} the Weierstrass points are shown (implicitly) to correspond to the edge-centres of the universal map. This shows that the order 2 rotation that permutes the vertices and face-centres adjacent to an edge-centre will preserve the corresponding Weierstrass point. They can also be interpreted geometrically as pairwise-symmetric distributed along the edges of the small stellated dodecahedron. As Weber identifies the order-3 symmetry as the rotation about the axis through opposite vertices of the \textit{unstellated} dodecahedron, one might wonder from this picture whether Weierstrass points are fixed points of the order-3 permutations in the group.  This turns out not be the case and a counting argument helps elucidate: the small stellated dodecahedron has 12 faces, which in turn means we want to have $\frac{60}{12}=5$ Weierstrass points per face. Hence where three faces overlap there must be three Weierstrass points `stacked' there, which are invisibly permuted by the action.

Having now identified the Weierstrass points as some of the geometric points, we give a concrete result about what the Weierstrass points are. 
\begin{prop}[\cite{Edge1978}]\label{WPprop}
Bring's curve has 60 Weierstrass points, on which $\Aut(\mathcal{B})$ acts transitively. Letting $\{ \alpha,\beta,\gamma\}$ be the roots of the cubic $x^3+2x^2+3x+4$, these are given in the $\mathbb{P}^4$-model by $W_{ijk}$ where, for example,
\[
W_{345} = [1:1:\alpha:\beta:\gamma]  . 
\]
\end{prop}
\begin{proof}
Edge, working in the $\mathbb{P}^4$-model, identifies the Weierstrass points with the 60 intersections of the curve with the 10 planes $\Pi_{ij} = \pbrace{ x_i = x_j}$. To do this, Edge quotes \cite{Wiman1895} to show that these intersection points are stalls\footnote{A \bam{stall} of a plane curve is a point where the osculating (hyper)plane has a 4-point intersection \cite{Edge1971}.}, i.e. inflection points of certain linear series, and for the canonical embedding these are exactly the Weierstrass points \cite[p.~37]{Arbarello1985}. Simple algebra then gives the exact expression we write down. This viewpoint makes it clear that the Weierstrass points at the intersection with $\Pi_{ij}$ are preserved by the transposition\footnote{In Theorem 4 of \cite{Dye1995}, these transpositions are associated with the Brianchon points of the Clebsch hexagon $H$ from the proof of Proposition \ref{prop: HC valid model for Bring}.} $(ij)$ only, so have orbits of size $\frac{120}{2}=60$, and as the automorphism restricts to a permutation of the Weierstrass points, the action is then transitive.
\end{proof}
With our naming convention, note $W_{ijk}$ is defined by $x_i=\alpha$, $x_j = \beta$, $x_k = \gamma$. If we choose a different labelling of the roots of the cubic, this would give a different labelling of the Weierstrass points. The Weierstrass points split as $60 = 6 \times 10$, with 6 Weierstrass points being fixed by each of the 10 involutions in $S_5$. 

The property that the automorphism group acts transitively on the Weierstrass points is very rare, as characterised by the following result. 
\begin{theorem}[\cite{Laing2010}, Theorem 15]
If $X$ is a Riemann surface of genus $g>2$ with $g^3-g$ Weierstrass points on which $\Aut X$ acts transitively then either
\begin{itemize}
    \item $g=4$ and $X$ is Bring's curve,
    \item $g=3$ and $X$ is Klein's curve, or
    \item $g=3$ and $\Aut X \cong S_4$. 
\end{itemize}
\end{theorem}

\begin{remark}
As we have an explicit birational map from our plane model to our canonical embedding of the curve, we can get the explicit forms of the Weierstrass points in the HC-model using Edge's identification of the Weierstrass points in the $\mathbb{P}^4$-model. If one does not have this information, it is still possible to calculate the Weierstrass points using Sage and some educated guesswork. Using computer algebra, and the characterisation of Weierstrass points as zeros of the Wronskian determinant, one can check that in the HC-model the Weierstrass points have base coordinates at the 60 roots of the polynomial equations  
\begin{align*}
    0 &= x^{12} - 32x^{11} - 114x^{10} - 200x^9 + 100x^8 + 48x^7 - 936x^6 + 1728x^5 - 2000x^4 + 3200x^3 \\
    &\phantom{=} - 2624x^2 + 768x - 64  , \\
    0 &= x^{24} - 24x^{23} + 1306x^{22} - 2864x^{21} + 10096x^{20} - 32704x^{19} - 5704x^{18} - 41824x^{17}  \\
    &\phantom{=} + 43056x^{16} + 831616x^{15} + 837856x^{14} + 992256x^{13} + 2603136x^{12} + 1238016x^{11}  \\
    &\phantom{=} + 1560576x^{10} + 5584896x^9 + 3357696x^8 + 3838976x^7 + 5856256x^6 + 2543616x^5  \\
    &\phantom{=} + 2200576x^4 + 1355776x^3 + 454656x^2 + 65536x + 4096  , \\
    0 &= x^{24} + 56x^{23} + 1176x^{22} - 1784x^{21} - 3904x^{20} + 36096x^{19} + 12776x^{18} - 211904x^{17}  \\
    &\phantom{=} + 304736x^{16} + 431616x^{15} + 339456x^{14} - 1985664x^{13} - 625344x^{12} + 1034496x^{11}  \\
    &\phantom{=} + 3512576x^{10} - 584704x^9 - 3572224x^8 - 2018304x^7 + 3303936x^6 + 3055616x^5  \\ 
    &\phantom{=} + 1099776x^4 + 45056x^3 + 229376x^2 - 16384x + 4096  .
\end{align*}
One can check that the Galois group of each polynomial is $C_4 \times S_3$. We have seen that in the HC-model the automorphisms have coefficients in $\mathbb{Z}[\zeta]$, and so we know the splitting field must be an extension of $\mathbb{Q}[\zeta]$, which accounts for the $C_4$ factor in the Galois group. $S_3$ has a subgroup of order $2$, corresponding to an extension of degree-2, and a brute force calculation shows that we also wish to adjoin $i\sqrt{2}$. This has already nearly reduced the problem, and then one needs a small moment of inspiration to find the last thing to adjoin. Looking at \cite{Riera1992} then may inspire one to adjoin the real root of the polynomial $x^3+7x^2+8x+4$, say $\xi$, and this gives the full splitting field. We can solve the cubic explicitly using Cardano's formula to find 
\[
\xi = -\frac{1}{3}\pbrace{7+  \psquare{\sqrt[3]{145+30\sqrt{6}}+\sqrt[3]{145-30\sqrt{6}}}}  , 
\]
and as such we could also take our splitting field to be
\[
\mathbb{Q}\psquare{\zeta, i\sqrt{2}, \sqrt[3]{145+30\sqrt{6}}}  . 
\]
We observe by Cardano's formula that $\alpha, \beta, \gamma \in \mathbb{Q}[i\sqrt{2}, \sqrt[3]{30+15\sqrt{6}}]$, and this latter field is isomorphic to $\mathbb{Q}[i\sqrt{2}, \sqrt[3]{145+30\sqrt{6}}]$. With these expression for the Weierstrass points, and the explicit knowledge of the automorphism group as an action on affine coordinates, we can find explicitly the transposition that preserves a given Weierstrass point. 

Note that an alternative method for computing Weierstrass points algorithmically has been implemented in Magma \cite{Hess2002a}.
\end{remark}
We now have seen how our Weierstrass points satisfy characterisations 1 and 4 of Definition \ref{def: Weierstrass points}, and we complete the picture with the following result which we believe new. 
\begin{prop}\label{prop: function and divisor associated to WP}
Define the points $P_{345}$, $P^\prime_{345}$ corresponding to the Weierstrass point $W_{345}$ (and similarly for the other Weierstrass points) by 
\begin{align*}
    P_{345} &= [\delta^\prime : \delta : (-43\alpha^2 - 113\alpha - 92)\beta/112 \\
    &\phantom{=}- (13\alpha^2 - 27\alpha - 20)/28 : (43\alpha^2 + 25\alpha - 4)\beta/112 + (13\alpha^2 + 3\alpha - 24)/28 : 1] \\
    P^\prime_{345} &= [\delta : \delta^\prime : (-43\alpha^2 - 113\alpha - 92)\beta/112 \\
    &\phantom{=} - (13\alpha^2 - 27\alpha - 20)/28 : (43\alpha^2 + 25\alpha - 4)\beta/112 + (13\alpha^2 + 3\alpha - 24)/28 : 1].
\end{align*}
Here $\alpha$, $\beta$ were defined in Proposition \ref{WPprop} and $\delta$, $\delta^\prime$ are the roots of 
\[
x^2 - \psquare{\frac{(11 \alpha + 12)\beta + 4(3\alpha + 2)}{14}}x + \psquare{\frac{23(155\alpha + 388)\beta + 92(97\alpha + 172)}{6272}}  .
\]
Then there is a holomorphic differential $\nu_{345}$ on Bring's curve with divisor 
\[
4W_{345} + P_{345} + P_{345}^\prime  ,
\]
and a meromorphic function on Bring's curve with divisor
\[
P_{145} + P_{145}^\prime + P_{245} + P_{245}^\prime -4W_{345}  . 
\]
\end{prop}
\begin{proof}
Note that algorithmically it is possible to construct these through the work in \cite{Hess2002}, the tools for which are partially implemented in Sage but not with enough generality for us to use out-the-box. As such we need a different approach.

We know that for any Weierstrass point $W$ there is a hyperplane $H \subset \mathbb{P}^3$ intersecting the canonical embedding with multiplicity $g=4$ at $W$ (that is $H \cap \mathcal{B} = 4W + P + P^\prime$ for some $P, P^\prime \in \mathcal{B}$). Such a hyperplane gives a holomorphic differential $\nu_{ijk}$ with $(\nu_{ijk}) = 4W_{ijk} + P_{ijk} + P^\prime_{ijk}$. From \cite{Edge1978} we know that for Bring's curve the osculating plane\footnote{The \bam{osculating plane} at a point $P$ on a curve is the limiting plane through $P$, $P^\prime$, $P^{\prime \prime}$ as $P^\prime, P^{\prime \prime} \to P$ on the curve \cite{Salmon1862}.} at $W$ intersects 4 times, and so this is the plane we are looking for. As Edge gives a formula for the osculating plane (attributed to Hesse), we can explicitly calculate the remaining 2 intersections with the curve in terms of a polynomial roots. This gives us the first result for the divisor of the meromorphic differential.

Furthermore, from \cite{Edge1981} we have a tritangent plane which has intersection with Bring's curve
\[
2(W_{145} + W_{245} + W_{345})  . 
\]
We will discuss this plane (and others like it) more in \S\ref{subsec: tritangent planes}, but for now all we need is that this means there is a holomorphic differential $\omega_{45\alpha}^{(1)}$ on $\mathcal{B}$ with $(\omega_{45\alpha}^{(1)}) = 2(W_{145} + W_{245} + W_{345})$. As such we get the divisor of the meromorphic function 
\[
\pround{\frac{\nu_{145}\nu_{245}}{\pround{\omega_{45\alpha}^{(1)}}^2}} = P_{145} + P_{145}^\prime + P_{245} + P_{245}^\prime -4W_{345}  . 
\]
As we can calculate the formula for all these planes explicitly if we wish, we could (in principle) construct the corresponding function and differential. 
\end{proof}
\begin{remark}
We are able to verify the results in Proposition \ref{prop: function and divisor associated to WP} using the Abel-Jacobi map implemented in the Sage \cite{DisneyHogg2021}; see the notebooks. 
\end{remark}
%%%%%%%%%%%%%%%%%%%%%%%%%%%%%%%%%%%%%%%%%%%%%%%%%%%%%%%%
%%%%%%%%%%%%%%%%%%%%%%%%%%%%%%%%%%%%%%%%%%%%%%%%%%%%%%%%
\subsection{Vertices and Face-Centres}
Using Figure 2 of \cite{Braden2012} we are able to link the Petrie polygon to the R\&R model of the curve, and this gives us a concrete expression for the remaining geometric points. 
\begin{prop}\label{FCprop}
The face-centres are exactly the points fixed by an order-5 automorphism of the curve. They are given in the $\mathbb{P}^4$-model by the permutations of $[1:\zeta:\zeta^2:\zeta^3:\zeta^4]$. The corresponding points in the HC-model are the desingularizations of $V_k$, together with $a$, $b$, $c$, $d$, and 
\begin{enumerate}
    \item $[-2\zeta^3 - 2\zeta^2 : \zeta^3 + \zeta^2 - 1 : 1] = [\sqrt{5}+1 : -\sqrt{5}/2-3/2 : 1] = [2j : j^\prime-2: 1]$,
    \item $[-2\zeta^3 - 2 : -2\zeta^3 - \zeta - 1 : 1]$,
    \item $[-2\zeta^2 - 2\zeta : -\zeta^3 + \zeta + 1 : 1]$,
    \item $[2\zeta^3 + 2\zeta^2 + 2\zeta : \zeta^3 + 2\zeta^2 + 2\zeta + 1 : 1]$,
    \item $[-2\zeta^2 - 2 : \zeta^3 - \zeta^2 + \zeta : 1]$,
    \item $[2\zeta^3 + 2\zeta + 2 : \zeta^2 - \zeta + 1 : 1]$,
    \item $[-2\zeta^3 - 2\zeta : 2\zeta^3 + \zeta^2 + \zeta + 2 : 1]$
    \item $[2\zeta^3 + 2\zeta^2 + 2 : -\zeta^3 - \zeta^2 - 2 : 1] = [-\sqrt{5}+1 : \sqrt{5}/2-3/2:1] = [2j^\prime : j-2 : 1]$,
    \item $[2\zeta^2 + 2\zeta + 2 : -\zeta^3 - 2\zeta^2 - \zeta : 1]$,
    \item $[-2\zeta - 2 : -\zeta^2 - 2\zeta - 1 : 1]$,
\end{enumerate}
where $j, j^\prime$ are the roots identified by Dye mentioned in \S\ref{subsec: automorphism group}. 
\end{prop}
\begin{proof}
Certainly the order-5 rotation about a face-centre fixes that centre, and so a Riemann-Hurwitz counting argument gives us that the $24$ face-centres are the fixed points of order-5 automorphisms. It is a simple matter of
computation to verify the given expressions are fixed;
this may be done in Sage.
\end{proof}
The Galois group $\operatorname{Gal}(\mathbb{Q}[\zeta]/\mathbb{Q}) \cong C_4$ acts element-wise on the face-centres and the orbits partition the set of face-centres into 6 sets of 4. Each set of 4 face-centres form the vertices of a quadrilateral whose edges lie in $\pbrace{H_2=0}$ \cite{Edge1978}. Moreover, the faces of the dodecadodecahedron corresponding to the face-centres in each quadrilateral are parallel.

We have a similar result for the vertices. 
\begin{prop}\label{Vprop}
The vertices are exactly the points fixed by an order-4 automorphism of the curve. They are given in the $\mathbb{P}^4$-model by the permutations of $[1: i: -1: -i: 0]$.
\end{prop}
We could use our birational map as above to give the vertices in HC coordinates, but these are not very illuminating. For example, that vertex $[1:i:-1:-i:0]$ maps to
\[
[X : Y : Z] = [(3\zeta^3 - \zeta^2 + 2\zeta + 1)i + 3\zeta^3 + 3\zeta^2 + 5 : (2\zeta^3 + 2\zeta + 1)i + 2\zeta^3 + 2\zeta^2 + 2 : 1]
\]
though $[1:-1:i:-i:0]$ maps to $[1: (-1+i)/2: (-1+i)/2]$.
%%%%%%%%%%%%%%%%%%%%%%%%%%%%%%%%%%%%%%%%%%%%%%%%%%%%%%%%
%%%%%%%%%%%%%%%%%%%%%%%%%%%%%%%%%%%%%%%%%%%%%%%%%%%%%%%%
%%%%%%%%%%%%%%%%%%%%%%%%%%%%%%%%%%%%%%%%%%%%%%%%%%%%%%%%
%%%%%%%%%%%%%%%%%%%%%%%%%%%%%%%%%%%%%%%%%%%%%%%%%%%%%%%%
\section{Quotients by Subgroups}\label{sec: quotients by subgroups}

Both \cite{Riera1992} and \cite{Weber2005} consider the quotient of $\mathcal{B}$ by the action of subgroups of $S_5$.
In this section we shall study the various quotients of Bring's curve of nonzero genus and the relations between them, both clarifying and extending previous work. In \S\ref{subsec: genera of quotients} we will use the Riemann-Hurwitz theorem to describe possible quotients.
In \S\ref{subsec: relations between quotients} we shall note the various relationships we expect between the quotients just on group theoretic grounds, while in \S\ref{subsec: quotients by 4 and 2-2 cycle}-\S\ref{subsec: quotients by transposition} we turn to their explicit construction. In so doing we discover a number of curious isomorphisms beyond those expected. In \S\ref{subsec: summarising quotients} we summarise our calculations and relate them to known isogeny results.  Throughout we will use the following: for any subgroup $H$ of the automorphisms $\Aut(\mathcal{C})$ of a curve
 $\mathcal{C}$,  $H\le \Aut(\mathcal{C})$,
then the normaliser $N_{\Aut(\mathcal{C})}(H)$ acts on the $H$-orbits
and $N_{\Aut(\mathcal{C})}(H)/H \le \Aut(\mathcal{C}/H)$; if $g\in N_{\Aut(\mathcal{C})}(H) $ we will denoted by $\overline{g}$ the $H$-coset of $N_{\Aut(\mathcal{C})}(H)$ containing this.
The quotient curves of this section are summarised in Figures \ref{eq: structure of quotient curves} and \ref{eq: tracking fixed points}.

%%%%%%%%%%%%%%%%%%%%%%%%%%%%%%%%%%%%%%%%%%%%%%%%%%%%%%%%
%%%%%%%%%%%%%%%%%%%%%%%%%%%%%%%%%%%%%%%%%%%%%%%%%%%%%%%%

\subsection{Genera of Quotients}\label{subsec: genera of quotients}
Our first step is to know the topology of the quotients we are going to find. To this end, we give the following result. 
\begin{prop}
The data of the quotients of $\mathcal{B}$ by subgroups $\pangle{\sigma} \leq S_5$ is summarised by the following table:
\begin{center}
\begin{tabular}{|c|c|c|c|c|c|c|}
    \hline 
$\sigma$ & Example & $\left\lvert{cl(\sigma)}\right\rvert$ & $\left\lvert{\operatorname{Fix(\sigma)}}\right\rvert$ & Fixed Points& $N_{S_5}(\pangle{\sigma})/\pangle{\sigma}$ & $g\pround{\mathcal{B}/\pangle{\sigma}}$ \\ 
\hline \hline 
    $(12)$ & - &10& 6 & Edges &$S_3$ &1\\ 
    &&&&&& \\
    $(123)$ & $R S$ &20 &0 &-&$V_4$ &2\\ 
    &&&&&& \\
    $(12)(34)$ & $R$ &15& 2 & Vertices&$V_4$&2\\
    &&&&&& \\
    $(1234)$ & $U$ &30& 2 &Vertices &$C_2$&1\\ 
    &&&&&& \\
    $(123)(45)$ & - &20& 0 &- &$C_2$&0\\  
    &&&&&&\\
    $(12345)$ & $S$ &24& 4 & Face-centres &$C_4$&0 \\ 
    \hline
\end{tabular}
\end{center}
Here $cl(\sigma)$ is the conjugacy class of $\sigma$, $\operatorname{Fix}(\sigma) = \{ P \in \mathcal{B} \, | \, \sigma(P) = P \}$. 
\end{prop}
\begin{proof}
As conjugate elements yield isomorphic quotients we need only to give one $\sigma$ per conjugacy class.  The first three columns follow from the group theory we have previously shown in \S\ref{subsec: automorphism group}, the fourth and fifth follow from \S\ref{sec: geometric points} and the sixth is elementary group theory. The final column remaining is then a Riemann-Hurwitz argument for genus, which will we demonstrate for the quotient by $(2345)$ as in \cite{Riera1992}.

For general $\sigma$, denoting by $\pi$ the projection $\mathcal{B} \to {\mathcal{B}}/{\pangle{\sigma}} := H$, Riemann-Hurwitz says 
\begin{align*}
g_\mathcal{B}-1 &= (\deg \pi )(g_H-1) + \frac{1}{2} B  , 
\end{align*}
where $B$ is the degree of ramification of $\pi$, which in the case of a quotient by a group action corresponds to the fixed point structure of $\sigma$.

Consider $(24)(35)$. A fixed point of $\mathcal{B}$ under this, given by projective coordinates $x_i,  i=1, \dots, 5$, must have 
\[
(x_1 , x_4 , x_5 , x_2 , x_3) = (\lambda x_1 , \lambda x_2 , \lambda x_3 , \lambda x_4 , \lambda x_5)  . 
\]
From this we can see $\lambda = \pm 1$. Taking $\lambda=1$ gives no solutions, but taking $\lambda=-1$ one finds the equations 
\[
x_1 = 0, \quad x_1^2 + 2x_2^2 + 2x_3^2 = 0, \quad x_1^3=0,
\]
which gives the 2 fixed points $[0:1:\pm i:-1:\mp i]$. Hence we have \begin{align*}
 4-1 = 2(g_H-1) + \frac{1}{2}(1+1)   
\Rightarrow  g_H = 2  . 
\end{align*}
Moving now to $(2 3 4 5)$, thinking about the possible branching structure we get from Riemann-Hurwitz 
\[
3 = 4(g_H-1) + \frac{1}{2} \psquare{\sum_{P \in \operatorname{Fix}(\sigma)} 3 + \sum_{P \in \operatorname{Fix}(\sigma^2)\setminus \operatorname{Fix}(\sigma)} 1}  . 
\]
Fixed points of $(2345)$ will correspond to points such that 
\[
(x_1 , x_3 , x_4 , x_5 , x_2) = (\lambda x_1 , \lambda x_2 , \lambda x_3 , \lambda x_4 , \lambda x_5)  , 
\]
and we get the constraint $\lambda^4=1$. The equations of the curve become 
\begin{align*}
    x_1 + x_2 (1+ \lambda + \lambda^2 + \lambda^3) &= 0  , \\
    x_1^2 + x_2^2 (1 + \lambda^2 + 1+ \lambda^2) &= 0  , \\
    x_1^3 + x_2^3(1 + \lambda^3 + \lambda^2 + \lambda) &= 0  , 
\end{align*}
and checking the possible cases one finds that the only fixed points of $(2345)$  are $[0:1:\pm i : -1: \mp i]$. These are also the only fixed points of $(24)(35) = (2345)^2$ and so the second sum vanishes, hence we find $g_H=1$.
\end{proof}
R\&R \cite{Riera1992} provides a nice visual interpretation of the quotient by a 4-cycle, namely think of a sphere with 4 handles attached around an equator, then $(2345)$ is the cycle rotating these handles to each other by a quarter turn about the axis through the center of the equator. The fixed points are then where this axis intersects the sphere. 
Edge \cite{Edge1978}, citing Wiman, says that Bring's curve ``is, in ten different ways, in $(2,1)$ correspondence with a plane curve of genus 1".
The 10 $(2,1)$ correspondences noted by Edge are exactly the 10 quotients by a transposition giving a $2:1$ map $\mathcal{B} \to \mathcal{E}$, where $\mathcal{E}$ is an elliptic curve. Such a map is called a bielliptic structure, and Bring's curve is the unique genus-4 curve to have 10 such structures \cite{Casnati2005}. This table also lets us reconstruct the results about the gonality of Bring's curve from \cite{Gromadzki2010}. 

%%%%%%%%%%%%%%%%%%%%%%%%%%%%%%%%%%%%%%%%%%%%%%%%%%%%%%%%
%%%%%%%%%%%%%%%%%%%%%%%%%%%%%%%%%%%%%%%%%%%%%%%%%%%%%%%%
\subsection{Relations between Quotients}\label{subsec: relations between quotients}
We now consider the various relationships we might expect between the quotient curves of Bring's curve.
Recall (again from Riemann-Hurwitz) that if $\mathcal{C}$ is a curve of genus $g\ge2$ and $H\le \Aut(\mathcal{C})$ is such that
$|H|>4(g-1)$, a so called \lq large automorphism group', then the
genus of the quotient curve $g_{\mathcal{C}/H}=0$. For Brings's curve
we are therefore interested in subgroups of $S_5$ of order less than or equal to $12$, and as we have seen the quotient by a 5-cycle leads to genus $0$ we may exclude subgroups containing such. The relevant
conjugacy classes of subgroups are then\footnote{Note that here we are using the GAP convention for the dihedral group that $|{D_n}|=n$. This is the opposite convention to Sage which uses $|{D_{n}}|=2n$.}
\begin{enumerate}[(a)]
    \item $\pangle{(12)(34)}\cong C_2$, $\pangle{(12),(34)}\cong V_4$, $\pangle{(1324)}\cong C_4$, $\pangle{(1324),(12)}\cong D_8$. Each of these groups $H$ have the same normaliser: $N_{S_5}(H)=\pangle{(1324),(12)}\cong D_{8}$.
    \item $\pangle{(12)}\cong C_2$, $\pangle{(345)}\cong C_3$, $\pangle{(345),(12)}\cong C_6$, $\pangle{(345), (34)}\cong S_3$,  $\pangle{(345),(12)(34)}\cong  S_3'$, $\pangle{(345),(12),(34)}\cong D_{12}\cong S_3\times C_2$. Each of these groups $H$ have the same normaliser: $N_{S_5}(H)=\pangle{(345),(12),(34)}\cong D_{12}$.
    
    \item $\pangle{(12)(34), (13)(24)}\cong V_4$, $\pangle{(234),(12)(34)}\cong A_4$. Each of these groups $H$ have the same normaliser: $N_{S_5}(\pangle{(234),(12)(34)}) = \pangle{(234), (12), (34)} \cong S_4$.
\end{enumerate}
The groupings here are such that if $H$, $H'$ are from the same item then they share the same normaliser $N=N_{S_5}(H)=N_{S_5}(H')$. In particular each of $H$, $H'$ and $HH'=H'H$ are normal in $N$ and we have commutativity of the following diagram of quotients
\begin{center}
\begin{tikzpicture}[baseline= (a).base]
\node[scale=1.2] (a) at (0,0){
\begin{tikzcd}
\mathcal{B}\arrow[r, "H"]\arrow[d, " H' "'] \arrow[rd, "HH' "]&\mathcal{B}/H\arrow[d, " HH'/H "] \\
\mathcal{B}/H'\arrow[r, " HH'/H'" ']  &\mathcal{B}/HH'
\end{tikzcd}
};
\end{tikzpicture}
\end{center}
where the arrows are labelled by the symmetry being quotiented; by an isomorphism theorem we have of course that $HH'/H\cong H'/(H\cap H')$
and $HH'/H' \cong H/(H\cap H')$. If, say $H\le H'$, then $HH'=H'$ and one
side of this diagram collapses.

Now a Riemann-Hurwitz calculation shows that the each of the quotients
by $\pangle{(12),(34)}$,  $\pangle{(1324),(12)}$, $\pangle{(345), (34)}$ and $\pangle{(345),(12),(34)}$ is of genus $0$ and so not being considered. Thus from the preceding discussion we have the following relations amongst quotients for the subgroups of (a), (b) and (c):
\begin{center}

\begin{tikzcd}
\mathcal{B}\arrow[d, " (12)(34) " '] \arrow[rd, "(1324) "] 
& \\
\mathcal{C}\arrow[r, ]  &\mathcal{E}\\
&(a)
\end{tikzcd}
\qquad
% Tikz solution from https://tex.stackexchange.com/questions/254613/length-of-arrow-in-tikz-cd
\begin{tikzcd}[column sep = large]
&\mathcal{B}\arrow[r, " (12)"]\arrow[d, "(345)  "' ]
\arrow[dl, "S_3'" ']\arrow[rd, " C_6 "]&\mathcal{E}'\arrow[d, " (345) " ] \\
\mathcal{E}''&\mathcal{C}^\prime \arrow[r, " (12) " '  ] \arrow[l, "\overline{(12)(34)}" ] &\mathcal{E}'''\\
&&(b)
\end{tikzcd}
\qquad
\begin{tikzcd}
&\mathcal{B}\arrow[d, " V_4" ']
\arrow[dl, "(12)(34)" ']\arrow[rd, " A_4 "]& \\
\mathcal{C} \arrow[r]&\mathcal{E}\sp{(iv)} \arrow[r, " \overline{(234)} " '  ]  &\mathcal{E}\sp{(v)}\\
&&(c)
\end{tikzcd}
\end{center}
Here  $\mathcal{C}, \mathcal{C}^\prime$ are genus 2 curves and $\mathcal{E}$,\ldots,$\mathcal{E}\sp{(v)}$ elliptic curves that we cannot yet specify purely on group theoretic grounds. We turn now to their specification and indeed we shall find some interesting identities between them, which will then be summarised in \S\ref{subsec: summarising quotients}.

\subsection{Quotients by a 4- and 2,2-cycle}\label{subsec: quotients by 4 and 2-2 cycle}
Armed with the knowledge of the genus of the quotients we expect we shall now write them explicitly. We will begin with the 2,2-cycle corresponding to $U^2$ with the 4-cycle $U$ arising in the discussion.  Recall that we have  
$$
U^2:[X,Y,Z]\rightarrow [X,Z,Y], \quad \psi(U^2)=(12)(34)  .
$$
The normaliser of $\pangle{(12)(34)}$ in $S_5$ is $N_{S_5}(\pangle{(12)(34)})=\pangle{(12), (1324)}\cong D_8$. As we remarked earlier, $N_{S_5}(\pangle{(12)(34)})/\pangle{(12)(34)}\cong V_4$ are symmetries which remain when we go to the quotient $\mathcal{B}/\pangle{U^2}$ and we expect the quotient genus-2 curve to have (at least) this $V_4$ symmetry group. One of these involutions, which we will later see to be $\overline{(12)} = \overline{(34)}$,\footnote{The $\pangle{(12)(34)}$ cosets of the normaliser are $\pbrace{e, (1 2)(3 4)}$, $\pbrace{(12), (34)}$, $\pbrace{(13)(24), (14)(23))}$, and $\pbrace{(1324), (1423)}$.} is the hyperelliptic involution of the quotient curve and so further quotienting by this symmetry gives $\mathbb{P}^1$. Quotienting by the other two involutions will yield elliptic curves, and we will complete this construction now, both from the perspective of the HC-model, and the $\mathbb{P}^4$-model.

Starting in the HC-model, in order to get the first quotient $\mathcal{B}/\pangle{U^2}$ we express our curve in terms of the invariants of $U^2$: $X$, $T:=Y+Z$, and $V:=YZ$. Then
\begin{align*}
 0&= X(Y^5+Z^5) + (XYZ)^2 - X^4YZ -2(YZ)^3 ,   \\  &
=X(T^5-5 T^3 V+5 T V^2)+ X^2 V^2 -X^4 V-2 V^3 . 
\end{align*}
In $\mathbb{P}\sp{1,1,2}$ this is our genus-$4$ curve. Setting\footnote{Equally one may set $X=1$ but we note that setting $V=1$ yields a genus 4 curve as here $V$ is not of weight $1$. 
} $T=1$ and viewing 
$$
0=X(1-5V+ 5V^2)+X^2 V^2 -X^4 V-2 V^3
$$
as the affine part of a projective curve we have (after a not-very-illuminating transformation, for which Maple was used) the hyperelliptic curve
\begin{equation}
\label{qc1} 
\mathcal{C}_1:\  B^2 =A^6 +4 A^5 +10 A^3 + 4 A +1  .
\end{equation}
This is the genus-2 curve of \cite{Riera1992} with automorphism group $V_4$. Calculating the Igusa-Clebsch invariants (which give the $\overline{\mathbb{Q}}$-isomorphism class of a curve) and searching the LMFDB \cite{LMFDB} shows that this is a model for the modular curve $X_0(50)$, which we verifying by directly calculating a model for $X_0(50)$ over $\mathbb{Q}$ using \cite{Shimura1995}. The substitutions $A=(2+A^\prime)/(2-A^\prime)$, $B=4  B'/(2-A^\prime)^3$ make the $V_4$ symmetry clearer, giving
$$
(B^\prime)^2 + \psquare{(A^\prime)^6-5 (A^\prime)^4 -40 (A^\prime)^2 -80} = 0  ,
$$
where we have the hyperelliptic involution $B'\rightarrow -B'$ and the map $A^\prime \rightarrow -A^\prime$ generating $V_4$. We shall now be explicit about how these work.

As previously mentioned, quotienting $\mathcal{C}_1$ by either of the two non-hyperelliptic involutions yields an elliptic curve, with each involution having two fixed points from Riemann-Hurwitz. Quotienting by $(B',A^\prime)\rightarrow (B',-A^\prime)$ by introducing $A^{\prime \prime} = -(A^\prime)^2$ yields
\begin{equation}
\label{qe1} 
\mathcal{E}_1: (B')^2=(A^{\prime \prime})^3-5 (A^{\prime \prime})^2 -40 (A^{\prime \prime}) -80\ \text{ with $j$-invariant } \ j_{\mathcal{E}_1 }=-\frac{5\times 29^3}{2^5} = j(\tau_0).
\end{equation}
The fixed points are $(A^\prime, B^\prime)=(0, \pm\sqrt{-80})$, corresponding to the two points $(X,V)=([1\pm \sqrt{5}]/2, -1\pm \sqrt{5}/2)$, which are the images of the four vertices $[1: \pm i : \mp i : -1 : 0]$ and $[1: \pm i : -1 \mp i : 0]$ respectively, depending on sign. We recognise $\mathcal{E}_1$ to be the elliptic curve $E_1$ in \cite{Riera1992}.

We may also quotient $\mathcal{C}_1$ by
$(B', A^\prime)\rightarrow (-B', -A^\prime)$. To do this write $\mathcal{C}_1$ as 
$$
(B')^2 C^4 + \psquare{(A^\prime)^6-5 (A^\prime)^4 C^2 -40 (A^\prime)^2 C^4 -80 C^6} = 0  , 
$$
from where we see the same automorphism acts as $C \to -C$. Quotienting by this action by introducing $C^\prime = -C^2$ yields 
$$
(B')^2 (C^\prime)^2 =(A^\prime)^6-5 (A^\prime)^4 (C^\prime) -40 (A^\prime)^2 (C^\prime)^2 -80 (C^\prime)^3  . 
$$
Setting $A^\prime=1$ and taking $B^{\prime \prime} =  B^\prime C^\prime$ gives the standard elliptic form
\begin{equation}
\label{qe2} 
\mathcal{E}_2: (B^{\prime \prime})^2 =  1-5  (C^\prime) -40  (C^\prime)^2 -80 (C^\prime)^3 , \quad j_{\mathcal{E}_2 }=-\frac{25}{2} = j(5\tau_0). 
\end{equation}
The fixed point of this involution is $[B^\prime : A^\prime : C] = [1:0:0]$; this is a singular point where the desingularization corresponds to the 2 points at infinity, or correspondingly $(X,V)=(\mp i/2-1/2, 1/4)$, which are the images of the two vertices $[1:-1: \pm i : \mp i : 0]$. We recognise $\mathcal{E}_2$ to be $E_2$ in \cite{Riera1992}.\footnote{There is a typo in the constant term of R\&R's $E_2$ if it is to have the stated $j$-invariant.}

Next let us quotient the $\mathbb{P}^4$-model directly by the action of $(12)(34)$ and compare with these quotients just obtained from the HC-model.
To this end we introduce semi-invariants of the action of $(1324) = \psi(U)$, defined by 
\[
(s_1, s_2, s_3, s_4)^T = \begin{pmatrix}
    1 & 1 & 1 & 1 \\ 1 & -1 & i & -i \\ 1 & 1 & -1 & -1 \\ 1 & -1 & -i & i
    \end{pmatrix} (x_1, x_2, x_3, x_4)^T  . 
\]
These are constructed such that $U \cdot s_j = i^{j-1} s_j$, and so we have $U^2$-invariants $[s:t:u:v] := [s_1: s_2 s_4 : s_3 : s_4^2] \in \mathbb{P}^{1,2,1,2}$.\footnote{Note one could have instead taken $v=s_2^2$ for the last invariant, but the choice above happens to be better around the vertices that are the fixed points of the 4-cycle. Moreover, recognise that the invariants taken are defined over the field extension $\mathbb{Q}[i]$ and not $\mathbb{Q}$.} In terms of these invariants we have 
\begin{align*}
    H_2 &= \frac{1}{4} \pround{5s^2 + u^2 + 2t}  , \\
    H_3 &= \frac{3}{16} \pround{-5 s^3 + su^2 + \frac{t^2u}{v} + 2st + uv}  .  
\end{align*}
Eliminating $t$ from these equations, and setting $s=1$, we can use Maple to get the genus-2 curve in Weierstrass form
\begin{equation}
\label{qc1p} 
\mathcal{C}_1:\  y^2 = 100-25x^2-10x^4-x^6
\end{equation}
where $x=u$ and $y=-10+2uv$. The roots of the sextic here are a M\"obius transform of the curve $\mathcal{C}_1$ previously given (namely $x \mapsto \sqrt{20}/x$), so these curves are the same over $\mathbb{Q}[\sqrt{5}]$. Note the reason we see $\mathbb{Q}[\sqrt{5}]$ here is that it is the degree-2 subfield of $\mathbb{Q}[\zeta]$, the field required to have equivalence of the Hulek-Craig and $\mathbb{P}^4$-models of Bring's curve. We now aim to identify the $V_4$ of this curve described earlier   with the quotient of the normaliser  $\pangle{(12), (1324)}/\pangle{(12)(34)}$. One can check that $(12) : [s:t:u:v] \to [s:t:u:t^2/v]$. This fixes $x$, and so must be the automorphism $y \to -y$; that is $\overline{(12)} = \overline{(34)}$ is the hyperelliptic involution of
$\mathcal{C}_1$. Indeed one can check 
\begin{align*}
    y = -10 + 2uv &\mapsto -10 + 2\frac{t^2u}{v}  , \\
    &= -10 - 2\pround{-5 + u^2 + 2t + uv}  , \\
    &= -10 - 2\pround{-10 + uv + (5 + u^2 + 2t)}  , \\
    &= 10-2uv = -y  . 
\end{align*}
Likewise, $(1324) : [s:t:u:v] \to [s:t:-u:-v]$, and so leads to the automorphism $(x,y) \to (-x,y)$. Now if $\mathsf{x}=x^2$, the elliptic curve
$y^2 = 100-25\mathsf{x}-10\mathsf{x}^2-\mathsf{x}^3$ has $j$-invariant
$-25/2$ and so we may identify $\mathcal{E}_2 \cong \mathcal{B}/\pangle{(12)(34), (1324)} = \mathcal{B}/\pangle{(1324)}$.
Similarly 
$(13)(24) : [s:t:u:v] \to [s:t:-u:-t^2/v]$ leads to $(x,y) \to (-x,-y)$
and
we may identify $\mathcal{E}_1 \cong \mathcal{B} / \pangle{(12)(34), (13)(24)}$. Note when we compare with \cite{Riera1992}, we differ from R\&R in the identification of which quotient is being taken and their ascribing of $\tau_0$ and $5\tau_0$. Our results agree with those of \cite{Weber2005}, wherein the author describes an order-4 rotation (which he calls $\phi$, but we shall call $\tilde{U}$), and calculates the quotient by its action $T=\mathcal{B}/\pangle{\tilde{U}}$ to be
$$
T:\ y^2=4 x^3-75 x -1475,\quad j_T=-\frac{25}{2} = j(5 \tau_0)  .
$$ 
Note that our strategy of semi-invariants can be used directly to calculate the quotient by the 4-cycle $(1324)$, something we were unable to do in the HC-model because of the non-linearity of the automorphism $U$. To do so introduce new variables invariant under $(1324)$ (and so necessarily under $(12)(34)$) given by $u^\prime = uv$, $v^\prime = v^2$. These let us rewrite the defining equations of Bring's curve as 
\begin{align*}
    H_2 &= \frac{1}{4} \pround{5s^2 + \frac{(u^\prime)^2}{v^\prime} + 2t}  , \\
    H_3 &= \frac{3}{16} \pround{-5 s^3 + \frac{s(u^\prime)^2}{v^\prime} + \frac{t^2u^\prime}{v^\prime} + 2st + u^\prime}  ,  
\end{align*}
and we can apply Maple to find a Weierstrass form of the resulting elliptic curve. This is exactly the process of quotienting by $(x,y) \mapsto (-x,-y)$ as above, but in a different language.

Furthermore, from our previous investigation using group theory, we know that the quotient $\mathcal{B}/V_4$, where this $V_4$ is the one containing only 2,2-cycles, can further be quotiented by $(234)$ to give $\mathcal{B}/A_4$. Rather than attempt to construct this quotient using the invariants previously calculated on the 2,2-quotient, we step back and recall that the invariant ring $\mathbb{Q}[x_1, \dots, x_n]^{A_n}$ is generated by the symmetric polynomials $s_k = \sum_{i=1}^n x_i^k$ for $k=1, \dots, n$ and the Vandermonde polynomial $V = \prod_{1 \leq i < j \leq n} (x_j - x_i)$. As generators of the invariant algebra we know that there must be a relation between $V^2$ and the other generators as $V^2$ is an $S_n$-invariant, and the $S_n$-invariant algebra is generated by the $s_k$. Taking $n=4$ in the case of Bring's curve, the relations imposed from the curve are $s_2 + s_1^2 = 0 = s_3 - s_1^3$, and this gives for the additional relation, 
\begin{equation}\label{ref: quotient by 3cycle}
4 s_4^3 - \frac{373}{16} s_1^4 s_4^2 + \frac{431}{8} s_1^8 s_4 - \frac{701}{16} s_1^{12} + V^2 = 0 . 
\end{equation}
Setting $s_1=1$ (as we may do at all points on the quotient except those points coming from the vertices on Bring's curve) we see this is clearly an elliptic curve $\mathcal{E}_3$ with hyperelliptic involution $V \to -V$, and for which we can calculate the $j$-invariant to be $\frac{211^3 \times 5}{2^{15}} = j(3 \tau_0)$.\footnote{This is the curve 15-isogenous to $\mathcal{E}_2$ noted by Serre, see \S\ref{subsec: period matrix}.} Further quotienting by the hyperelliptic involution then corresponds to the quotient $\mathcal{B}/S_4$, which is $\mathbb{P}^1$ as expected, being the quotient by a large automorphism group.  
%%%%%%%%%%%%%%%%%%%%%%%%%%%%%%%%%%%%%%%%%%%%%%%%%%%%%%%%

%%%%%%%%%%%%%%%%%%%%%%%%%%%%%%%%%%%%%%%%%%%%%%%%%%%%%%%%
\subsection{Quotients by a 3-cycle}\label{subsec: quotient by 3-cycle}
We may utilise the same methods illustrated above for the 2,2-cycles to calculate the quotient of Bring's curve by a 3-cycle. We will work with the 3-cycle $(345)$ for purely aesthetic resaons. 
Now the normaliser of $\pangle{(345)}$ in $S_5$ is $N_{S_5}(\pangle{345}))=\pangle{(12), (34), (345)}\cong D_{12}$ and
we expect the quotient genus-2 curve $\mathcal{B}/\pangle{(345)}$ to have $D_{12}/C_3 \cong V_4$ symmetry group.
Again one of these involutions, which we will later see to be $\overline{(34)}$,\footnote{The cosets of $\pangle{(345)}$ in the normaliser are now $\pbrace{e, (345), (3 5 4)}$,
 $\pbrace{(12), (12)(345), (12)(3 5 4)}$,
 $\pbrace{(34), (45), (35))}$, and 
$\pbrace{(12)(34), (12)(45), (12)(35))}$.} is the hyperelliptic involution on the curve and so further quotienting by this symmetry gives $\mathbb{P}^1$. Quotienting by the other two involutions again yields elliptic curves we shall now describe. As the quotients of both the HC-model and $\mathbb{P}\sp4$-model via semi-invariants proceed analogously we shall present here only the $\mathbb{P}\sp4$-model calculations.

Letting $\rho$ be a primitive cube-root, we take semi-invariants of the action of $(345)$
\[
(s_1, s_2, s_3, s_4)^T = \begin{pmatrix}
    1 & 0 & 0 & 0 \\ 0 & 1 & 1 & 1 \\ 0 & 1 & \rho & \rho^2 \\ 0 & 1 & \rho^2 & \rho
    \end{pmatrix} (x_2, x_3, x_4, x_5)^T  . 
\]
We correspondingly take invariants $[s:t:u:v] = [s_1 : s_2 : s_3 s_4 : s_3^3] \in \mathbb{P}^{1,1,2,3}$.\footnote{Note in this case that the invariant $v$ is not defined over $\mathbb{Q}$, but over the cyclotomic extension $\mathbb{Q}[\rho]$.} In terms of these variables we have 
\begin{align*}
    H_2 &= \frac{2}{3} \pround{3s^2 + 3st + 2t^2 + u}  , \\
    H_3 &= \frac{1}{9} \pround{-27s^2t - 27st^2 - 8t^3 + v + 6tu + \frac{u^3}{v}}  .
\end{align*}
Eliminating $u$ from these and setting $s=1$, we can use Maple to get the genus-2 curve in Weierstrass form 
\begin{equation}
\label{qc2} 
\mathcal{C}_2 :  y^2 = 108\pround{4+12x+95x^2+170x^3+155x^4+72x^5+16x^6}
\end{equation}
where $x=t$ and $y=-90t-90t^2-40t^3+4v$. Examining the roots of the sextic confirms that $\mathcal{C}_2$ is a genuinely distinct genus-2 curve from $\mathcal{C}_1$. Indeed, this curve does not currently exist in the LMFDB\footnote{This is to be expected, as it does not satisfy the conditions to be included in curves investigated in \cite{Booker2016}}, but one can check that over $\mathbb{Q}[\sqrt{5}]$ it is isomorphic to the curve 2500.a.400000.1 given by $y^2 = -7x^6 - 8x^5 + 10x^3 - 8x - 7$. Moreover, introducing $x^\prime, y^\prime$ by $x = {-2}/{(1+x^\prime)}$, $y = {y^\prime}/{(1+x^\prime)^3}$, yields
\[
(y^\prime)^2 = 432\psquare{(x^\prime)^6 + 80(x^\prime)^4 + 125(x^\prime)^2 + 50}  . 
\]
This makes the $V_4$ symmetry evident, being generated by $x^\prime \to -x^\prime$ and $y^\prime \to -y^\prime$. The elliptic curve obtained from quotienting by the action $x^\prime \to -x^\prime$ is 
\begin{equation}
\label{qe3} 
\mathcal{E}_4 :  (y^\prime)^2 = 432\psquare{(x^{\prime\prime})^3 + 80(x^{\prime\prime})^2 + 125x^{\prime\prime}+ 50},  \quad j_{\mathcal{E}_4 }=-\frac{5^2\times 241^3}{2^3} = j(15\tau_0)  . 
\end{equation}
To our knowledge this elliptic curve has not been previously noted in discussions of Bring's curve. The elliptic curve obtained from quotienting by the action $(x^\prime, y^\prime) \to (-x^\prime, -y^\prime)$ is the previously seen
\[
\mathcal{E}_1 :  (y^{\prime \prime})^2 = 432\psquare{1 + 80 z^\prime + 125(z^\prime)^2 + 50(z^\prime)^3}  , \quad j_{\mathcal{E}_1} = -\frac{5 \times 29^3}{2^5} = j(\tau_0)  .
\]

We now wish to identify the quotient $\pangle{(12), (34), (345)}/\pangle{(345)}\cong V_4$
with the  $V_4$ just described. It requires a little effort to see $(12) : [s : t : u : v] \to [-s - t : t : u: v]$, which corresponds to $x^\prime \to -x^\prime$, and that $(34) : [s : t : u : v] \to [s : t : u: u^3/v]$, which fixes $x^\prime$ and so must be the map $y^\prime \to -y^\prime$. This means that the remaining involution $(x^\prime, y^\prime) \to (-x^\prime, -y^\prime)$ comes from the group element $(12)(34)$. As such we identify $\mathcal{E}_4 \cong \mathcal{B} / \pangle{(12), (345)}$, and we have a curious isomorphism, namely 
$$
\mathcal{B} / \pangle{(12)(34), (345)} \cong \mathcal{B} / \pangle{(12)(34), (13)(24)} \cong \mathcal{E}_1
%$$
\text{ corresponding to}\ 
%\begin{center}
\begin{tikzcd}
\mathcal{B}\arrow[r, "(345)"]\arrow[d, "(12)(34)"] &\mathcal{C}_2\arrow[d, "\overline{(12)(34)}"] \\
\mathcal{C}_1\arrow[r, "\overline{(13)(24)}" ']  &\mathcal{E}_1
\end{tikzcd}.
%\end{center}
$$

As before we can also consider distinguished points on the quotients coming from fixed points. Whereas $(345)$ has no fixed points when acting on $\mathcal{B}$ each involution on $\mathcal{C}_2$ that gives a quotient to an elliptic curve has 2 fixed points. The fixed points of $\overline{(12)}$ are the orbits under $(345)$ of the six fixed points of $(12)$ that are Weierstrass points. The fixed points of $\overline{(12)(34)}$ are the orbits under $(345)$ of the $3 \times 2$ fixed points of $(12)(34)$, $(12)(45)$, and $(12)(35)$ that are vertices.

%%%%%%%%%%%%%%%%%%%%%%%%%%%%%%%%%%%%%%%%%%%%%%%%%%%%%%%%
\subsection{Quotients by a transposition}\label{subsec: quotients by transposition}
The strategy of quotienting from the $\mathbb{P}^4$-model using semi-invariants from the previous sections can also be used to calculate the quotient of Bring's curve by a transposition. We take semi-invariants
\[
(s_1, s_2, s_3, s_4)^T = \begin{pmatrix}
    1 & 1 & 0 & 0 \\ 1 & -1 & 0 & 0 \\ 0 & 0 & 1 & 1 \\ 0 & 0 & 1 & -1
    \end{pmatrix} (x_1, x_2, x_3, x_4)^T  , 
\]
and in terms of these variables the defining equations of Bring's curve become 
\begin{align*}
H_2&=  \frac12( s_1^2+s_2^2)+ s_3^2+s_4^2+ (s_1+s_3+s_4)^2  , \\
H_3&=  \frac14( s_1^3+3 s_1 s_2^2)+ s_3^3+s_4^3 -(s_1+s_3+s_4)^3  .    
\end{align*}
Taking invariants $[s:t:u:v] = [s_1 : s_2^2 : s_3 : s_4]$ and eliminating $t$ yields the elliptic curve
\begin{equation}
\label{qe2p} 
\mathcal{E}_2: \ 4s^3 + 8s^2u + 7s u^2 + u^3 + s v^2 - u v^2=0, \quad j_{\mathcal{E}_2}=-\frac{25}{2} = j(5 \tau_0).
\end{equation}
We observe that (also curiously) the quotient of $\mathcal{B}$ by a transposition is isomorphic to the quotient by a 4-cycle. It is not clear that there is any apriori reason for this to be the case. 
We are however able to relate the elliptic curves $\mathcal{E}_2$ and $\mathcal{E}_4$ as follows.

The normaliser of $\pangle{(12)}$ in $S_5$ is $N_{S_5}(\pangle{(12)})=\pangle{(12),(34),(345)}\cong D_{12}$ where now
$N_{S_5}(\pangle{(12)})/\pangle{(12)}\cong S_3$.\footnote{The cosets of $\pangle{(12)}$ in the normaliser are $\pbrace{e, (12)}$,
$\pbrace{(345), (12)(345)}$,
$\pbrace{(354), (12)(354)}$,
$\pbrace{(34), (12)(34)}$,
$\pbrace{(35), (12)(35)}$, and $\pbrace{(45), (12)(45)}$. }
We may immediately identify the action of $\overline{(34)}$ as the hyperelliptic involution on $\mathcal{E}_2$ as it acts as $v \to -v$, but we also retain another group action, that of $\overline{(345)}$.
To understand this action recall that there are six fixed points of the action of $(12)$ on $\mathcal{B}$, which are all Weierstrass points. This gives six more distinguished points on the curve $\mathcal{E}_2$ in addition to the two images of the vertices we have seen previously.  
The six Weierstrass points fixed by $(12)$ break into two orbits of three
under $(345)$
which we denote by $\{H_i\}$, $\{H_i'\}$ ($i=1,2,3$). 
Just as $(345)$ gives Bring's curve as an unramified cover of the genus $2$ curve $\mathcal{C}_2$, quotienting each by a transposition has $\overline{(345)}$  yielding an unramified automorphism of the quotient curves. Hence there exists a quotient from $\mathcal{E}_2=\mathcal{B}/\langle(12)\rangle$ to another elliptic curve
$\mathcal{E}_4 = \mathcal{B}/\pangle{(12),(345)}$ such that the following diagram commutes,
\begin{center}
\begin{tikzcd}
\mathcal{B}\arrow[r, "(345)"]\arrow[d, "(12)"] &\mathcal{C}_2\arrow[d, "\overline{(12)}"] \\
\mathcal{E}_2\arrow[r, "\overline{(345)}" ']  &\mathcal{E}_4
\end{tikzcd}
\end{center}
To view this action on the elliptic curve $\mathcal{E}_2$, we calculate that $(345): [s:t:u:v] \to [s:t: -s-u/2-v/2 : s+3u/2-v/2]$. In the $s=1$ affine chart $4 + 8u + 7 u^2 + u^3 + (1 - u) v^2=0$ we can take as a cohomology basis 
\[
{\eta:=\frac{du}{2v(1-u)}} ,
\]
and it is a simple algebraic calculation to see that this differential is invariant under the action of $\overline{(345)}$. If $\{\overline{H}_i\}$,
$\{\overline{H'}_i\}$
are the images of the Weierstrass points on $\mathcal{E}_2$ such that
${(345)}({H}_i)={H}_{i+1}$, $\overline{(345)}(\overline{H}_i)=\overline{H}_{i+1}$ and so forth, the invariance of $\eta$ tells us that $$\int_{\overline{H}_1}\sp{\overline{H}_{2}}\eta=
\int_{\overline{H}_2}\sp{\overline{H}_{3}}\eta=\ldots,\quad
\int_{\overline{H'}_1}\sp{\overline{H'}_{2}}\eta=
\int_{\overline{H'}_2}\sp{\overline{H'}_{3}}\eta=\ldots,\quad
\int_{\overline{H}_1}\sp{\overline{H'}_{1}}\eta=
\int_{\overline{H}_2}\sp{\overline{H'}_{2}}\eta=\ldots
$$
This means that if $\Lambda=\pangle{1, 5 \tau_0}$ is the period lattice of the elliptic curve
$\mathcal{E}_2$ then $\int_{\overline{H}_i}\sp{\overline{H}_{i+1}}\eta$
and
 $\int_{\overline{H'}_i}\sp{\overline{H'}_{i+1}}\eta$
are the same fixed element of $\Lambda/3$. By choosing an appropriate basis we
may take $\overline{(345)} : z \to z+1/3$ for $z\in\mathbb{C}/\pangle{1, 5 \tau_0}$. Then
the quotient map will be a 3:1 isogeny of elliptic curves and we 
have the period of $\mathcal{E}_4$ to be $15 \tau_0$ and $j_{\mathcal{E}_4} = j(15 \tau_0)$.
Under this isogeny the six images on $\mathcal{E}_2$ of the Weierstrass points $[1:1:\alpha:\beta:\gamma]$ are mapped to two on $\mathcal{E}_4$.
Indeed, one may construct the isogeny explicitly using the Weierstrass-$\wp$ function and the Abel-Jacobi map via the diagram 
\begin{center}
	\begin{tikzcd}
    \mathcal{E}_2 \arrow[r] \arrow[d,"P \mapsto \int_\infty^P \eta"'] & \mathcal{E}_4 \\
    J(\mathcal{E}_2) \arrow[r, "z \mapsto 3z"'] & J(\mathcal{E}_4) \arrow[u, "z \mapsto (\wp(z) {,}\wp^\prime(z))"']
	\end{tikzcd}
\end{center} 
where $J(\mathcal{E})$ represents the Jacobian of the elliptic curve viewed at $\mathbb{C}/\pangle{1,\tau}$.

\begin{remark}
This calculation may be verified numerically, as shown in the notebooks.
\end{remark}

\begin{remark}
Note that when we call $\overline{(345)}$ an automorphism of the elliptic curve, we have to think about it only in the category of projective varieties with morphisms given by rational maps. To consider the group structure on an elliptic curve $E$ we need to also give a base point $O \in E$ which acts as the additive identity. A morphisms from pair $(E, O) \to (E^\prime, O^\prime)$ must map $O \to O^\prime$. If we denote the group of automorphisms of $E$ as a projective variety as $\Aut(E)$, and the group of $O$-fixing automorphisms as $\Aut_O(E)$, then we have short exact sequence 
\[
0 \to T_E \to \Aut(E) \to \Aut_O(E) \to 0  ,
\]
where $T_E$ is the group of translations of $E$. It is a classical theorem that $\Aut_O(E)\in\{C_2, C_4, C_6\}$ (here we are over $\mathbb{C}$). 

Here $\overline{(345)}$ is indeed a translation. To see this take variables $x,y$ defined by 
\begin{align*}
    u &= \frac{-70s^3 + 6sx}{2(25s^2+3x)},\quad
    v = \frac{-3y}{25s^2 + 3x} ,
\end{align*}
which in the affine patch where $s=1$ give $\mathcal{E}_2$ as the curve $x^3-25/3 x+2950/27+y^2 = 0$. In these coordinates $(345)$ acts as 
\begin{align*}
    x &\mapsto  \frac{5(275 - 3x + 15y)}{3(65 + 15x - 3y)},\quad
    y \mapsto \frac{20(55 - 15x - 3y)}{(65 + 15x - 3y)} .
\end{align*}
Indeed for the projective coordinates $[X:Y:Z]$ of our curve
$X^3-25/3 X Z^2 +2950/27 Z^3+Y^2 Z = 0$ we have under $\overline{(345)}$ that
\begin{align*}
[X:Y:Z]&\mapsto [{5(275Z/3 - X + 5Y)}:20(55Z - 15X - 3Y):65Z + 15X - 3Y ]\\
&\mapsto [275Z/3 - X - 5Y: -220Z + 60X - 12Y: 13Z + 3X + (3Y)/5  ] \mapsto [X:Y:Z].
\end{align*}
When working with a Weierstrass model of an elliptic curve it is standard to take the distinguished point to be $\infty=[0:1:0]$ and we see in terms of $(x,y)$ that  $\infty\mapsto (-25/3, 20)\mapsto (-25/3, -20)\mapsto \infty$. 
\end{remark}

%%%%%%%%%%%%%%%%%%%%%%%%%%%%%%%%%%%%%%%%%%%%%%%%%%%%%%%%
%%%%%%%%%%%%%%%%%%%%%%%%%%%%%%%%%%%%%%%%%%%%%%%%%%%%%%%%
\subsection{Summarising}\label{subsec: summarising quotients}
We can collect the information of the quotients we have seen into the following diagram. Solid arrows represent a covering map coming from a quotient, whereas `squiggly' arrows indicate isomorphisms that are unexplained by group theory alone. 

% https://tex.stackexchange.com/questions/223026/enumerating-tikz-cd-diagrams
\begin{center}
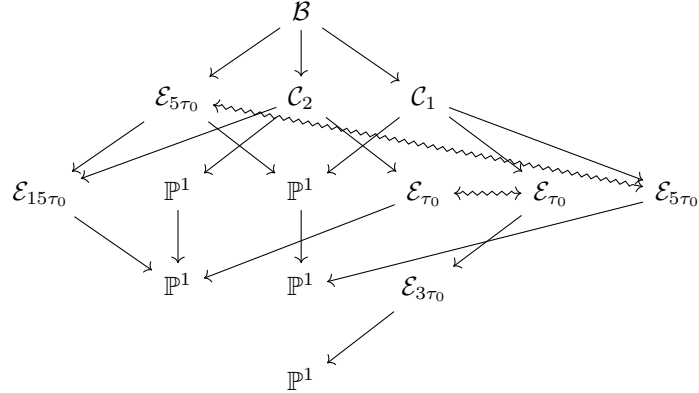
\begin{figure}[H]
    \begin{tikzcd}
        & & \mathcal{B} \arrow[dl] \arrow[d] \arrow[dr] & & & \\
        & \mathcal{E}_{5 \tau_0} \arrow[dl] \arrow[dr] \arrow[drrrr, leftrightsquigarrow] & \mathcal{C}_2 \arrow[dll] \arrow[dl] \arrow[dr] & \mathcal{C}_1 \arrow[dl] \arrow[dr] \arrow[drr] & & \\
        \mathcal{E}_{15 \tau_0} \arrow[dr] & \mathbb{P}^1 \arrow[d] & \mathbb{P}^1 \arrow[d] & \mathcal{E}_{\tau_0} \arrow[dll] \arrow[r, leftrightsquigarrow] & \mathcal{E}_{\tau_0} \arrow[dl] & \mathcal{E}_{5 \tau_0} \arrow[dlll] \\
        & \mathbb{P}^1 & \mathbb{P}^1 & \mathcal{E}_{3 \tau_0} \arrow[dl] & & \\
        & & \mathbb{P}^1 & & &
    \end{tikzcd}
    \caption{Quotient structure of Bring's curve.}
    \label{eq: structure of quotient curves}
\end{figure}
\end{center}

This corresponds to the quotients by the following groups, where an arrow now indicates that a source group is normal in the target, whereas a dashed line just indicates that a group is a subgroup. The label used corresponds to the list in \S\ref{subsec: relations between quotients}, except where that label is ambiguous and the exact group must be specified.

\begin{center}
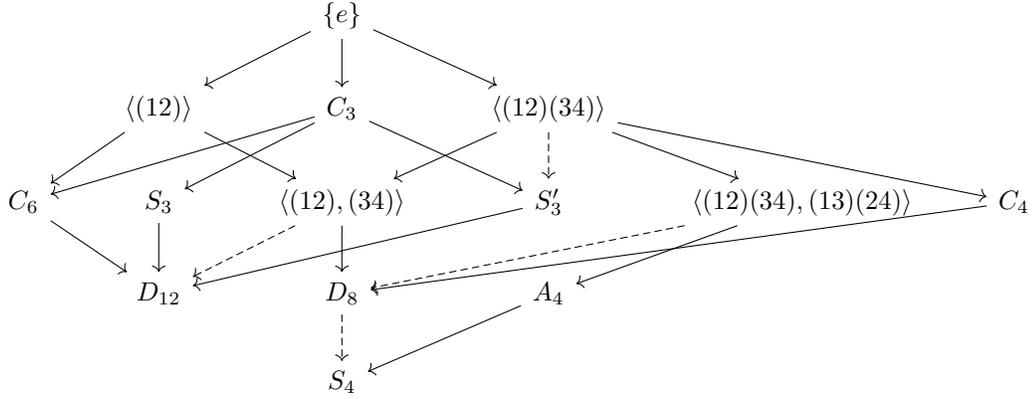
\begin{figure}[H]
    \begin{tikzcd}
        & & \pbrace{e} \arrow[dl] \arrow[d] \arrow[dr] & & & \\
        & \pangle{(12)} \arrow[dl] \arrow[dr] & C_3 \arrow[dll] \arrow[dl] \arrow[dr] & \pangle{(12)(34)} \arrow[dl] \arrow[d, dashed] \arrow[dr] \arrow[drr] & & \\
        C_6 \arrow[dr] & S_3 \arrow[d] & \pangle{(12),(34)} \arrow[dl, dashed] \arrow[d] & S_3^\prime \arrow[dll] & \pangle{(12)(34), (13)(24)} \arrow[dll, dashed] \arrow[dl] & C_4 \arrow[dlll] \\
        & D_{12} & D_8 \arrow[d, dashed] & A_4 \arrow[dl] & & \\
        & & S_4 & & &
    \end{tikzcd}
    \caption{Subgroup structre of $S_5$ corresponding to the quotients of Bring's curve.}
    \label{eq: structure of quotient groups}
\end{figure}
\end{center}

Finally, we recall from a Riemann-Hurwitz argument that each elliptic curves that comes a quotient from a genus-2 curve with $V_4$ symmetry has 2 marked points which are the branch points of the covering map, equivalently the images of the fixed points of the involution being quotiented by. We have shown that the preimages of these on Bring's curve are geometric points. In Figure \ref{eq: tracking fixed points} we illustrate the quotient structure highlighting these points. (We do not decorate the curve $\mathcal{B}/A_4 \cong \mathcal{E}_{3\tau_0}$, but include it so as to show every quotient with genus greater than 0.) Note the nodes of this diagram do not correspond directly to the placements of those above. 

\begin{center}
\begin{figure}[H]
    \begin{tikzcd}
        & & 6\textcolor{red}{\times}, 4\textcolor{blue}{\otimes}, 2\textcolor{blue}{\ocircle}, 4\textcolor{blue}{{\square}} {\includegraphics[height=0.05\textwidth]{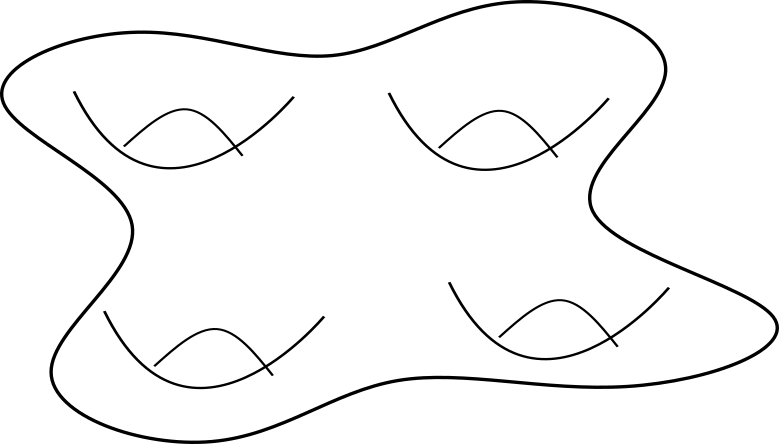}}  \arrow[dl] \arrow[d] \arrow[dr] & & & \\
        & 6\textcolor{red}{\times} {\includegraphics[height=0.05\textwidth]{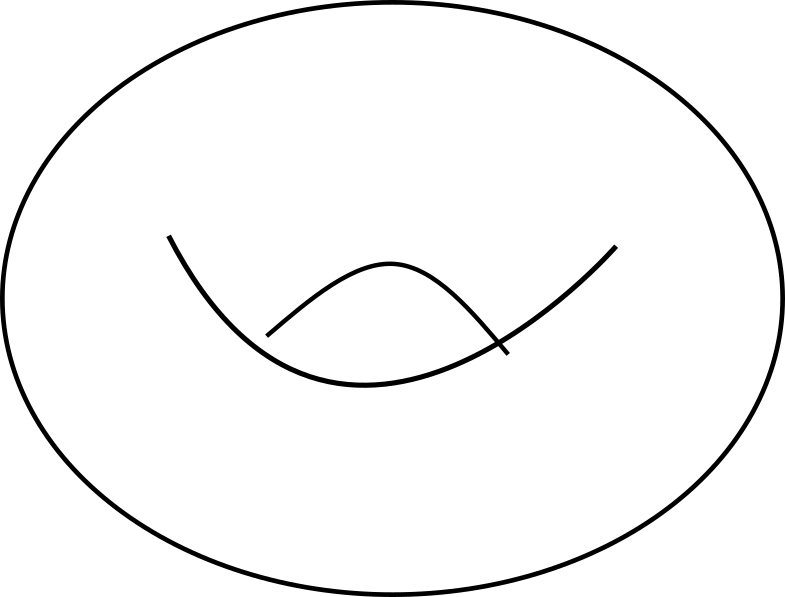}}  \arrow[d]  & 2\textcolor{red}{\times},  2\textcolor{blue}{\ocircle} \includegraphics[height=0.05\textwidth]{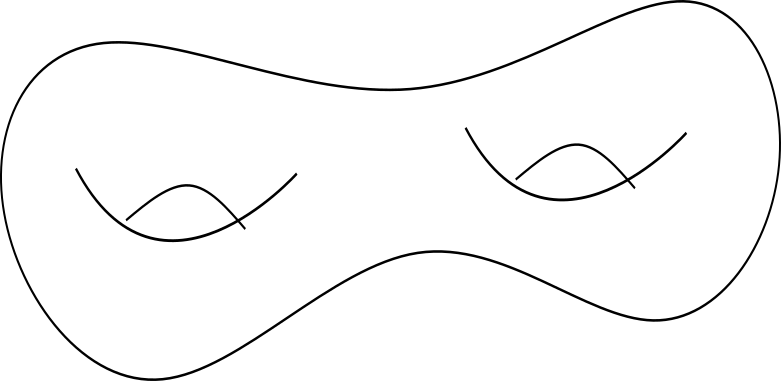}  \arrow[dl]  \arrow[d] & 2\textcolor{blue}{\ocircle}, 2\textcolor{blue}{{\square}} {\includegraphics[height=0.05\textwidth]{g2.png}} \arrow[d] \arrow[dr] & & \\
        & 2\textcolor{red}{\times} {\includegraphics[height=0.05\textwidth]{g1.png}}  &   2\textcolor{blue}{\ocircle} {\includegraphics[height=0.05\textwidth]{g1.png}}  & 2\textcolor{blue}{\square} {\includegraphics[height=0.05\textwidth]{g1.png}} \arrow[d] & 2\textcolor{blue}{\ocircle} {\includegraphics[height=0.05\textwidth]{g1.png}} & \\
        & & & \phantom{2\textcolor{blue}{\square}}{\includegraphics[height=0.05\textwidth]{g1.png}} & &
    \end{tikzcd}
\caption{Quotient structure including marked points. The points $\textcolor{red}{\times}$ correspond to the Weierstrass points $[1: 1: \alpha : \beta: \gamma]$ (and permutations of $\alpha, \beta, \gamma$), the points $\textcolor{blue}{\otimes}$ to the vertices $[1: -1: 0: \pm i : \mp i]$ and $[1: -1: \pm i : 0 : \mp i]$, the points $\textcolor{blue}{\ocircle}$ to the vertices $[1: -1: \pm i : \mp i : 0]$, and the points $\textcolor{blue}{{\square}}$ correspond to the vertices $[1: \pm i : \mp i : -1 : 0]$ and $[1: \pm i : -1 : \mp : 0]$.}
\label{eq: tracking fixed points}
\end{figure}
\end{center}

With the information of the quotients, we can now discuss the isogeny class and isomorphism class of the Jacobian of Bring's curve with the following results. 

\begin{prop}[\cite{Riera1992}, \S4, \cite{Weber2005}, Corollary 5.5]\label{prop: CC-isogeny class of Bring's Jacobian}
The $\mathbb{C}$-isogeny class of the Jacobian of Bring's curve is $\mathcal{E}_3^4$.
\end{prop}
\begin{proof}
As we will later want to strengthen this result, we will use methods which relate subvarieties of the Jacobian to idempotents, and thus to subgroups of the automorphism group. Namely we will use \cite[\S4.2]{Lange2001}, which gives the isogeny decomposition in terms of Prym varieties of the Jacobian of a curve with $A_5$ action. Note, because of comments made before, Bring's curve is the unique genus-4 curve for which we could apply this argument. The result follows as, using the notation of Lange, $X_{A_4} = \mathcal{E}_3$, $X_{D_5} = X_{Z_5} = Y = \mathbb{P}^1$.

This proof also follows from \cite[Proposition 5.1]{Lange2001}, which uses only the action of $S_4$ on the curve. Moreover, one could use \cite[Theorem C]{Kani1989} taking the subgroups to be $\pangle{(12)}, \pangle{(34)}$, $C_4$ and $A_4$. Alternatively, using the isomorphism $A_5 \cong PSL_2(\mathbb{F}_5)$ and \cite[Example 2]{Kani1989} on can find $\Jac \mathcal{B} \sim \Jac \mathcal{C}_1 \times \Jac\mathcal{C}_2$ and proceed from there. These proof strategies all follow the same approach of looking for idempotents.

One may also obtain the result following the same method as R\&R. We use that fact that isogenies act on the period matrix by right multiplication by matrices $R \in M_{4}(\mathbb{Z})$, and so we have the required isogeny by taking $\lambda=1$ in the identity
$$
\begin{pmatrix} \lambda\sp{-1}\id \end{pmatrix} \begin{pmatrix} 1 & \tau M \end{pmatrix}  \begin{pmatrix} \lambda\id & 0 \\ 0 & M^\prime \end{pmatrix} = \begin{pmatrix}\id & \dfrac{5}{\lambda}\tau \id \end{pmatrix}  ,
$$
where
$$
M':= 5 M^{-1} = \begin{pmatrix} 2 & -1 & 1 & -1 \\ -1 & 2 & -1 & 1 \\ 1 & -1 & 2 & -1 \\ -1 & 1 & -1 & 2 \end{pmatrix}  .
$$
\end{proof}

Note that we are able to construct the quotient $\mathcal{B}/A_4$ directly from the $\mathbb{P}^4$-model using transformations with coefficients in $\mathbb{Q}$, and the resulting curve is defined over $\mathbb{Q}$, so by \cite[Remark 6]{Kani1989} the above proposition can be strengthened to a statement about the $\mathbb{Q}$-isogeny class of the Jacobian of Bring's curve. Calculating using Sage we find that the $\mathbb{Q}$-isogeny class of $\mathcal{B}/A_4$ is 50a using the Cremona labels for elliptic curve $\mathbb{Q}$-isogeny classes. Hence the following result holds. 

\begin{prop}[\cite{Serre2008}, Exercise 8.3.2(b)]\label{prop: JB Q-isogeny class}
The $\mathbb{Q}$-isogeny class of the Jacobian of the $\mathbb{P}^4$-model of Bring's curve is $(50a)^4$. 
\end{prop}
Note in the above proposition we had to be careful to specify the Jacobian of the $\mathbb{P}^4$-model of Bring's curve, as we have seen that HC-model is not birationally equivalent over a field that doesn't contain $\mathbb{Q}[\zeta]$. The $\mathbb{Q}$-isogeny class of the elliptic curve $\mathcal{B}/\pangle{(12)(34), (13)(24)}$ calculated via the HC-model is 50b. Similarly, the $\mathbb{Q}$-isogeny class of the two elliptic curves covered by $\mathcal{C}_2$ is 450b, and the computation of the quotient required the coefficient field to be $\mathbb{Q}[\rho]$. In order to not have a contradiction with Proposition \ref{prop: JB Q-isogeny class}, we must have that the $\mathbb{Q}$-isogeny classes 50b and 50a merge over $\mathbb{Q}[\sqrt{5}]$, and that the isogeny classes 450b and 50a merge over $\mathbb{Q}[\rho]$, which is indeed the case.\footnote{The isogeny classes 450b and 50b merge over $\mathbb{Q}[\sqrt{-15}]$.}

One could use computational tools such as those in \cite{Booker2016, Lombardo2018}\footnote{To make the reconstruction process of Lombardo more accessible, we recreated in Sage some of the functions implemented by Lombardo in Magma.} to numerically find the $\mathbb{Q}$-isogeny class of the Jacobian of the HC-model of Bring's curve, as we did in the notebooks. Such computational results using idempotents can be helpful for developing our understanding, for example one can use computer algebra to search for relations between the characters $\operatorname{Ind}_H^{S_5}(1_H)$, that is the characters of $S_5$ induced from the trivial representation of $H \leq S_5$. Doing so gives relations between subvarieties of the Jacobian of Bring's curve following \cite[Theorem 3]{Kani1989}, for example 
\begin{equation}
\begin{aligned}\label{isogenies}
    J_{\pangle{(12)}} \times J_{\pangle{(12)(34), (13)(24)}} & \sim J_{\pangle{(12)(34)}} \times J_{\pangle{(12), (34)}} , \\
    J_{\pangle{(12)(34), (13)(24)}} \times J_{S_3} &\sim J_{\pangle{(12), (34)}} \times J_{S_3^\prime} .
\end{aligned}
\end{equation}
Using Riemann-Hurwitz arguments these would let us say that $\mathcal{B}/\pangle{(12)} \sim \mathcal{B}/C_4$ and $\mathcal{B}/\pangle{(12)(34), (13)(24)} \sim \mathcal{B}/S_3^\prime$ without having to do any calculation, entirely from group theory. The reason why these isogenies are actually isomorphisms is not clear.

Proposition \ref{prop: CC-isogeny class of Bring's Jacobian} can also be strengthened in a different direction.

\begin{prop}[\cite{Gonzalez2000}]
The Jacobian of Bring's curve is isomorphic as a complex torus to $\mathcal{E}_2^3 \times \mathcal{E}_1$.
\end{prop}
\begin{proof}
The proof in \cite{Gonzalez2000} is very general, considering Jacobians whose period matrix is invariant under $S_n$ for any $n$. The isomorphism from the period matrix given in \cite{Riera1992} is
$$
C \begin{pmatrix}
1 & \tau M
\end{pmatrix} \begin{pmatrix}
D & 0 \\ 0 & E
\end{pmatrix} = \begin{pmatrix}1, \tau \operatorname{diag}(5, 5, 5, 1) \end{pmatrix}
$$ 
where
$$
C = \begin{pmatrix}
-1 & 1 & -1 & 2 \\ 1 & 0 & 0 & 1 \\ 0 & -1 & 0 & 1 \\ 0 & 0 & 0 & 1
\end{pmatrix},  D = \begin{pmatrix}
0 & 1 & 0 & -1 \\ 0 & 0 & -1 & 1 \\ -1 & -1 & -1 & 4 \\ 0 & 0 & 0 & 1
\end{pmatrix},  E = \begin{pmatrix}
-1 & 1 & 0 & 0 \\ 1 & 0 & -1 & 0 \\ -2 & -1 & -1 & 1 \\ 1 & 0 & 0 & 0
\end{pmatrix}  .
$$
\end{proof}
Note it would not be possible that the Jacobian is isomorphic to the product of the elliptic curves as a principally polarised abelian variety, as it is well know that the Jacobian of any smooth compact Riemann surface with canonical principal polarisation is irreducible. In particular this means the matrix $\begin{pmatrix}
D & 0 \\ 0 & E
\end{pmatrix}$ of the proposition is not symplectic.

%%%%%%%%%%%%%%%%%%%%%%%%%%%%%%%%%%%%%%%%%%%%%%%%%%%%%%%%
%%%%%%%%%%%%%%%%%%%%%%%%%%%%%%%%%%%%%%%%%%%%%%%%%%%%%%%%
\section{Theta Characteristics}\label{sec: theta characteristics}
We now investigate the theta characteristics of Bring's curve, which in \cite{Atiyah1971} were shown to be equivalent to spin structures, and as such we will use the terms interchangeably. These are objects which have been studied greatly in a variety of contexts in algebraic geometry, number theory, and even string theory, though in this paper we will not discuss these applications and focus on calculation. Among other results we shall identify the unique invariant spin structure of Bring's curve.

For completeness we recall the main definition that we shall use in this paper.
\begin{definition}
A \bam{theta characteristic} is a half-canonical divisor class, i.e. a divisor class $[D]$ such that $2[D] = [\mathcal{K}_{\mathcal{B}}]$. 
\end{definition}
\begin{lemma}
Call a theta characteristic odd/even based on the parity of it's index of speciality. Then on a Riemann surface of genus $g$ there are $2^{g-1}(2^g-1)$ odd theta characteristics and $2^{g-1}(2^g+1)$ even theta characteristics. 
\end{lemma}
There are many equivalent definitions of parity for a theta characteristic, and while the above is not computationally the easiest, it is one of the easier to state. \cite{Farkas2012} provides a nice overview of these and their connections. 
\begin{example}\label{example: theta characteristic on Bring's curve}
It is a computational exercise to verify $\mathcal{K}_\mathcal{B} \sim \pround{\frac{-x v_3 v_4}{v_2}} = 2(3a+b-c)$, and so $\Delta = 3a+b-c$ is a theta characteristic on Bring's curve. A simple calculation in Sage shows that it is even. 
\end{example}

%%%%%%%%%%%%%%%%%%%%%%%%%%%%%%%%%%%%%%%%%%%%%%%%%%%%%%%%
%%%%%%%%%%%%%%%%%%%%%%%%%%%%%%%%%%%%%%%%%%%%%%%%%%%%%%%%
\subsection{Abel-Jacobi and the Riemann Constant Vector}\label{subsec: AJ and RCV}
We will now take a detour to lay out some notation and definitions in preparation for later.
\begin{definition}\label{def: AJ map and RCV}
The \bam{Abel-Jacobi map} on a Riemann surface $\mathcal{C}$ with basis of differentials $\pbrace{\omega_j}$, based at $Q \in \mathcal{C}$, is given by 
\[
(\mathcal{A}_Q(P))_j = \int_Q^P \omega_j  .
\]
This is well defined up to periods on the curve. Letting $\tau$ be the Riemann matrix of $\mathcal{C}$, the \bam{Riemann Constant Vector (RCV)} based at $Q$ is given by 
\[
(K_Q)_j = -\frac{1+\tau_{jj}}{2} + \sum_{k \neq j}^g \oint_{a_k} \omega_k(P) (\mathcal{A}_Q(P))_j  .
\]
\end{definition}
The RCV is related to the canonical divisor by $2K_Q\equiv\mathcal{A}_Q(\mathcal{K}_\mathcal{C})$. Indeed, \cite{Deconinck2015} shows how to find $K_Q$ from this relation given the theta function. We then have the following theorem about the divisor of the Sz\"ego kernel $\Delta_S$, defined in \cite{Fay1973}. 
\begin{theorem}[\cite{Fay1973}]
$K_Q = \mathcal{A}_Q(\Delta_S)$, and as such $\Delta_S$ is a theta characteristic.
\end{theorem}
The definition of the RCV also requires that the differentials are normalised with respect to the homology basis taken $\pbrace{a_j , b_j}$ such that $\oint_{a_j} \omega_i = \delta_{ij}$. As such hidden in the above result is the fact that the Sz\"ego kernel divisor is homology dependent, because the RCV is.

%%%%%%%%%%%%%%%%%%%%%%%%%%%%%%%%%%%%%%%%%%%%%%%%%%%%%%%%
%%%%%%%%%%%%%%%%%%%%%%%%%%%%%%%%%%%%%%%%%%%%%%%%%%%%%%%%
\subsection{Tritangent Planes and Odd Characteristics}\label{subsec: tritangent planes}
Recall that when we consider the canonical embedding of a curve in $\mathbb{P}^{g-1}$, the intersection with a hyperplane gives an effective element of the canonical divisor class of the curve, as it corresponds to the zero-locus of a differential on the curve. A tritangent plane is a plane that intersects the curve tangentially in three points. If a hyperplane is tangent at each intersection (i.e. the intersection is of order 2) then this naturally constructs an effective theta divisor on the curve, and moreover every such theta characteristic $\theta$ gives a hyperplane \cite{Celik2019}. For any sextic that is the intersection of a smooth quadric and smooth quartic in $\mathbb{P}^3$, for example Bring's curve, the 120 tritangent planes are in 1-1 correspondence with the 120 odd theta characteristics \cite[Theorem 2.2]{Harris2017}.

On Bring's curve we have the following result about these planes.
\begin{prop}[\cite{Edge1981}]
The 120 tritangent planes on Bring's curve split into two classes, 60 in each class;
\begin{enumerate}
    \item those where all three contact points are Weierstrass points, and 
    \item those where only one contact points is a Weierstrass point. 
\end{enumerate}
Planes in the first and second class respectively have equations (recalling notation from \S\ref{subsec: Weierstrass points})
\begin{align*}
\Pi_{\alpha  jk }^{(1)} &:= \pbrace{\frac{x_j}{\beta} - \frac{x_k}{\gamma} = 0 } , \\
\Pi_{ijk}^{(2)} &:= \pbrace{(\alpha-1)(\alpha+4)x_i + (\beta-1)(\beta+4)x_j + (\gamma-1)(\gamma+4)x_k =0 } , 
\end{align*}
where $[x_i] \in \mathbb{P}^4$ and $i,j,k$ distinct.
\end{prop}
To clarify the notation we have used for the planes, recall that the position of the indices on Weierstrass points $W_{ijk}$ indicates which root was equal to $x_i$. The same principle holds for $\Pi_{\alpha jk}^{(1)},  \Pi^{(1)}_{i \beta k},$ and $\Pi^{(1)}_{ij \gamma}$. 
\begin{corollary}\label{cor: Orbit decomposition of odd theta characteristics}
The orbit decomposition of odd theta characteristics on Bring's curve is \begin{equation}\label{eq:oddcharacteristics}
120=20+20+20+60 . 
\end{equation}
\end{corollary}
\begin{proof}
The characteristics coming from the triangent planes are 
\begin{align*}
T_{\alpha jk}^{(1)} &= \sum_{\substack{i=1 \\ i \neq j,k}}^5 W_{ijk}  , \\
T_{ijk}^{(2)} &= W_{ijk} + O_{ijk}^+ + O_{ijk}^-  . 
\end{align*}
Here we define the points $O_{ijk}^\pm$ by, for example, 
\[
O_{345}^\pm = \left [ \frac{1 \pm i\sqrt{15}}{2} : \frac{1 \mp i\sqrt{15}}{2} : \alpha^2 + \alpha + 1, \beta^2 + \beta + 1 , \gamma^2 + \gamma + 1 \right ] .
\]
A simple orbit-stabiliser argument then gives the orbit decomposition as, for example, $T_{\alpha 45}^{(1)}$ is stabilised by the symmetric group $S_{\{ 1,2,3 \}}$ and $T_{345}^{(2)}$ is stabilised by $S_{\{1,2\}}$. 
\end{proof}
\begin{remark}
While we were able to proceed analytically here characterising the orbits of the odd characteristics using the  work of earlier authors who had identified the tritangent planes of the curve we also may obtain the orbit structure numerically using Sage. Given a homology representation of the automorphism group, the algorithm to compute the orbits is exact, but the current implementation in Sage requires floating point arithmetic to calculate this representation. 
\end{remark}
%%%%%%%%%%%%%%%%%%%%%%%%%%%%%%%%%%%%%%%%%%%%%%%%%%%%%%%%
%%%%%%%%%%%%%%%%%%%%%%%%%%%%%%%%%%%%%%%%%%%%%%%%%%%%%%%%
\subsection{Even Characteristics}
In the previous subsection, we were able to fully characterise the odd theta characteristics on Bring's curve without too much difficulty through the use of existing work giving the stalls of the canonical embedding. The story for even characteristics is different, as the line bundle corresponding to a generic even theta characteristic has no sections, which is an obstacle to their study \cite{Takagi2009}. In \cite{Dolgachev1993} the authors were able to use the Scorza correspondence to calculate the orbit decomposition of even spin structures on Klein's curve, but at the time of writing the theory of the Scorza map has not been developed far enough to cover Bring's curve. The Scorza quartic necessary to extend Dolgachev and Kanev's work beyond $g=3$, wherein it becomes a codimension-1 hypersurface in $\mathbb{P}^{g-1}$, has been proven to exist for a generic even spin curve (that is a pair consisting of a curve and an even theta characteristic on it); it has only been explicitly found for trigonal curves in $g=4$ \cite{Takagi2009}. In \cite{Burns1983}, motivated by its realisation via an elliptic modular surface, Burns identified a theta characteristic on Bring's curve invariant under the $A_5$ subgroup of the automorphism group, though he described this characteristic only in terms of two line bundles on the curve, not directly in terms of points on the curve. We shall now fully classify the orbits of the even characteristics and give an explicit description of Burns' divisor in the process.

We have two distinct methods of probing the orbit decomposition of the theta characteristics on Bring's curve.
\begin{enumerate}
    \item Use the method of \cite{Kallel2006}, wherein theta characteristics are identified with vectors in $\mathbb{F}_2^{2g}$, and the action of automorphisms given by the homology representation of the automorphism group of the curve as found using the methods of \cite{Bruin2019}. 
    \item Identify theta characteristic with the $2^{2g}$ translates by half-lattice vectors of a half-canonical vector in the Jacobian of the curve, done using the implementation of the Abel-Jacobi map developed by Disney-Hogg \cite{DisneyHogg2021}, and the action of automorphisms given by the cohomology representation of the automorphism group of the curve as found using the methods of \cite{Bruin2019}.
\end{enumerate}
The two methods have different strengths, namely that the first uses exact computations rather than approximations with multi-precision arithmetic, but the second can concretely relate theta characteristics from their representation to an actual divisor (class). Both methods not only verify Corollary \ref{cor: Orbit decomposition of odd theta characteristics} but also gives the following result. 
\begin{theorem}
The orbit decomposition of even theta characteristics on Bring's curve is 
\begin{equation}\label{eq:evencharacteristics}
136 = 1 + 5+5+5+10+10+10+30+30+30  . 
\end{equation}
\end{theorem}
\begin{corollary}
Bring's curve has a unique theta characteristic invariant under the action of the automorphism group, which is also the theta characteristic invariant under the $A_5$ subgroup found in \cite{Burns1983}.
\end{corollary}

The existence of the unique invariant theta characteristic was known in \cite{Braden2012}, but it had not been identified. This we rectify with the following result. 

\begin{theorem}\label{invarianttheta}
The theta characteristic $\Delta$ (defined in Example \ref{example: theta characteristic on Bring's curve}) is the unique invariant theta characteristic in on Bring's curve. 
\end{theorem}
\begin{proof}
We first consider the action of $S$ on $a$, $b$, $c$, $d$. We have
\begin{align*}
a:=[0:0:1]&\simeq[2t^3:t:1]\xrightarrow{S}[2t^3:\zeta t:\zeta^{-1}]=[2 \zeta t^3:\zeta^2 t:1]
=[2(\zeta^2 t)^3:\zeta^2 t:1]\\& =[2\epsilon^3:\epsilon:1],\\
 b:=[0:1:0]&\simeq[2t^2:1/t:1]\xrightarrow{S}[2t^2:\zeta/ t:\zeta^{-1}]=
 [2 \zeta t^2:\zeta^2/ t:1]=[2(t/{\zeta^2})^2:\zeta^2/ t:1]\\ 
 c:=[1:0:0]_2&\simeq[1:t:t^4]\xrightarrow{S}[1:\zeta t:\zeta^{-1} t^4]=[1:\zeta t:(\zeta t)^4],\\
 d:=[1:0:0]_1&\simeq[1:t^4:t]\xrightarrow{S}[1:\zeta t^4:\zeta^{-1} t]=[1:(\zeta^{-1} t)^4: \zeta^{-1} t].
\end{align*}
Thus $a$, $b$, $c$ and $d$ are invariant under the symmetry $S$ and consequently $\Delta= 3a+b-c$ is also invariant.

Similarly, as mentioned in Proposition \ref{prop: U is an automorphism}, the action of $U$ on $a,b,c,d$ can be calculated as 
\[
a \mapsto c \mapsto b \mapsto d \mapsto a  ,
\]
and so 
\[
\Delta  = 3a+b-c \mapsto 3c+d-b = 3a+b-c - (3a+2b-4c-d) = \Delta - (x) \sim \Delta  .
\]
We have thus shown that $\Delta$ is invariant under $\pangle{S, U}$. To complete the proof that $\Delta$ is the invariant theta characteristic, one could attempt to show that $\Delta$ is invariant under the action of $R$ by direct computation, but this proves to be difficult. It is instead better to check in Sage that the unique spin structure invariant under the whole automorphism group is actually also the unique spin structure invariant under the subgroup generated by $S$ and $U$. In fact, by Theorem 1.2 of \cite{Kallel2006} and our work in \S\ref{sec: quotients by subgroups}, we know there is a unique theta characteristic invariant under $\pangle{S}$\footnote{Under the action of $S$, the orbit decompositions are $120=24\times 5$ and $136 = 1+27\times 5$. }, and this completes the proof. 
\end{proof}
This proof strategy is similar to the identification of the unique invariant theta characteristic on Klein's curve in \cite{Kallel2006}.\footnote{The final part of the proof in that paper, showing invariance under the order-2 generator of $PSL_2(\mathbb{F}_7)$, is in our opinion incomplete; Kallel (private correspondence) believes our reasoning correct. The theorem remains correct nevertheless, as one can check using their methodology that invariance under the order-7 generator is enough to specify the unique invariant spin structure, as the decompositions are $28=4\times 7$ and $36 =1 + 5 \times 7$. It is a curious coincidence that in both cases it was a generator of order 2 for which the direct calculation was difficult.} 

We now want to make the connection to \cite{Burns1983}. Recall Lemma \ref{quaddif} and Proposition \ref{prop: quadric of canonical embedding isomorphic to PxP}. This gives us two degree-3 maps $f_i : \mathcal{B} \to \mathbb{P}^1$, namely $f_i = \left. \pi_i \circ \varphi^{-1}\right\rvert_{\mathcal{B}} $, $i=1,2$, where $\varphi$ was the isomorphism $\mathbb{P}^1 \times \mathbb{P}^1 \to \mathcal{Q}$, and $\pi_i$ the projection to the two factors of $\mathbb{P}^1 \times \mathbb{P}^1$. What are the corresponding divisors? Working in the $L_a$ coordinates, we can use Sage to find that $f_1^{-1}([1:0]) = 2[0:0:0:1]+[1:0:0:0] = 2b+c := L^\prime$, while $f_2^{-1}([1:0]) = 2[0:1:0:0]+[0:0:0:1]= 2d+b := L$.
\begin{prop}\label{burnsLs}
The divisors $L$, $L^\prime$ satisfy the properties described in \cite{Burns1983}, namely 
\begin{align*}
    \Delta &\sim 3(L^\prime -L) + L, \quad 
    \mathcal{K}_\mathcal{B} \sim L+L^\prime, \quad  
    0 \sim 5(L^\prime - L)  . 
\end{align*}
\end{prop}
\begin{proof}
This is straightforward verification from the definitions. 
\end{proof}
Note we can connect this back to the degree-3 map given by Klein. We know from \cite{Weber2005} that we expect this map to be branched at face-centres of the $\pbrace{5, 5  |  3}$ tessellation, and these come from face-centres of the $\pbrace{5,4}_6$ tessellation, and indeed recall we saw that $a,b,c,d$ were face-centres. 

An additional characterisation of the invariant theta characteristic is given by the following result. 
\begin{prop}\label{szegoinvt}
In the homology basis of \cite{Riera1992}, the RCV satisfies \[
K_a = \frac{1}{10}(
3 , 2 , -2 , -3
) + \Im(\tau_0)(
1 , -2 , -2 , 1
)i  = \mathcal{A}_a(\Delta) . 
\]
As such, in the R\&R homology basis, the unique invariant theta characteristic is the divisor (class) of the Sz\"ego kernel, i.e. $\Delta = \Delta_S$. 
\end{prop}
\begin{proof}
The first equality is shown analytically in \cite{Braden2012}. The second is shown numerically in the corresponding notebooks, using the Abel-Jacobi map developed by Disney-Hogg \cite{DisneyHogg2021}. To verify the RCV, we implemented the methodology of \cite{Deconinck2015} using the theta function in Sage developed by Bruin and Ganjian \cite{Bruin2021}. While these calculations are numerical in nature, the calculations can be done with arbitrary binary precision. We were satisfied by calculating with 400 binary digits of precision, giving an absolute error of less than $10^{-118}$ for the first equality, and of less than $10^{-23}$ for the second equality.  
\end{proof}

%%%%%%%%%%%%%%%%%%%%%%%%%%%%%%%%%%%%%%%%%%%%%%%%%%%%%%%%
%%%%%%%%%%%%%%%%%%%%%%%%%%%%%%%%%%%%%%%%%%%%%%%%%%%%%%%%
%%%%%%%%%%%%%%%%%%%%%%%%%%%%%%%%%%%%%%%%%%%%%%%%%%%%%%%%
%%%%%%%%%%%%%%%%%%%%%%%%%%%%%%%%%%%%%%%%%%%%%%%%%%%%%%%%
\section{Conclusion}\label{sec: conclusion}
We have now seen that Bring's curve has many interesting properties and appears in a number of quite different settings; with this background now established we can summarise the main new results of the paper. We have:
\begin{enumerate}
    \item Constructed an explicit birational map between the HC- and $\mathbb{P}^4$-model of the curve, Proposition \ref{prop: HC valid model for Bring}.
    \item Constructed an explicit transform (\ref{transRRW}) between the period matrix of Weber and that of R\&R.
    \item Provided a complete realisation of the automorphism group of the curve in the HC-model (Propositions \ref{autA5}
    and
    \ref{prop: U is an automorphism}) and the canonical model (\ref{can12},\ref{can145}).
    \item Unified the functional, algebraic and geometric pictures of the distinguished orbits on the curve (those of size 24, 30, and 60) by explicitly identifying these points (Propositions \ref{WPprop},\ref{FCprop},\ref{Vprop}) and identifying the divisors of the holomoprhic differentials and meromorphic functions associated to the Weierstrass points (Proposition \ref{prop: function and divisor associated to WP}).
    \item Completed the classification of, and relations between, the quotients of Bring's curve (Figures \ref{eq: structure of quotient curves} and \ref{eq: structure of quotient groups}), finding a previously unidentified elliptic curve in the process (\ref{ref: quotient by 3cycle}). We identify one of the known quotient curves with the modular curve $X_0(50)$.
    \item Proven in a new way the $\mathbb{Q}$-isogeny class of the Jacobian
    (Proposition \ref{prop: JB Q-isogeny class}) and established a number of isogenies between (products of) the various Prym varieties associated with the quotients (\ref{isogenies}).
    \item Computed the orbit decomposition of both the odd and even theta characteristics on the curve (\ref{eq:oddcharacteristics},\ref{eq:evencharacteristics}) and determined the unique invariant theta characteristic (Proposition \ref{invarianttheta}) identifying this with the Sz\"ego divisor (Proposition \ref{szegoinvt}) and a divisor identified by Burns (Proposition \ref{burnsLs}).
\end{enumerate}
At various points we have corrected the literature.

In addition to our written account there are the accompanying notebooks.
Throughout 
we have provided examples of how modern computational tools can simplify calculation and provide insight to aid discovery. The
notebooks utilise SageMath, Maple, and Macaulay2 (all interfaced through Sage) and the code may be easily modified for other curves.

Finally we would draw attention to two areas that (for us) merit attention. First are the isomorphisms in Figure \ref{eq: structure of quotient curves} we have been unable to explain within the quotient structure of the curve. Perhaps this is something that the modular theory of the Bring's curve can shed light on. 
The second area deals with the orbit decompositions of the  theta characteristics. In the case of even characteristics our construction was purely numeric and not analytic: a solution to this  likely involves extending the theory of the Scorza quartic. Further, for both even and odd characteristics, there is a \lq\lq threeness \rq\rq we observe that we have found no satisfying reason for. We look forward to the clarification of these.

%%%%%%%%%%%%%%%%%%%%%%%%%%%%%%%%%%%%%%%%%%%%%%%%%%%%%%%%
%%%%%%%%%%%%%%%%%%%%%%%%%%%%%%%%%%%%%%%%%%%%%%%%%%%%%%%%
%%%%%%%%%%%%%%%%%%%%%%%%%%%%%%%%%%%%%%%%%%%%%%%%%%%%%%%%
%%%%%%%%%%%%%%%%%%%%%%%%%%%%%%%%%%%%%%%%%%%%%%%%%%%%%%%%
%%%%%%%%%%%%%%%%%%%%%%%%%%%%%%%%%%%%%%%%%%%%%%%%%%%%%%%%
%%%%%%%%%%%%%%%%%%%%%%%%%%%%%%%%%%%%%%%%%%%%%%%%%%%%%%%%

\providecommand{\bysame}{\leavevmode\hbox
to3em{\hrulefill}\thinspace}
% \bibliographystyle{amsalpha}
% \bibliography{extracted.bib}

\newcommand{\etalchar}[1]{$^{#1}$}
\providecommand{\bysame}{\leavevmode\hbox to3em{\hrulefill}\thinspace}
\providecommand{\MR}{\relax\ifhmode\unskip\space\fi MR }
% \MRhref is called by the amsart/book/proc definition of \MR.
\providecommand{\MRhref}[2]{%
  \href{http://www.ams.org/mathscinet-getitem?mr=#1}{#2}
}
\providecommand{\href}[2]{#2}

\end{document}